\documentclass[a4paper,11pt,reqno]{amsart}

\usepackage{enumerate}
\usepackage{amssymb, amsmath}
\usepackage{mathrsfs}%,dsfont}
\usepackage{amscd} \usepackage[active]{srcltx} \usepackage{verbatim}
\usepackage[colorlinks,linkcolor={blue},citecolor={blue},urlcolor={black}]{hyperref}

\theoremstyle{plain}
\newtheorem{theorem}{Theorem}[section]
\newtheorem{corollary}[theorem]{Corollary}
\newtheorem{lemma}[theorem]{Lemma}
\newtheorem{proposition}[theorem]{Proposition}

\newtheorem{definition}[theorem]{Definition}
\newtheorem{assumption}[theorem]{Assumption}

\newtheorem*{definition*}{Definition}

\theoremstyle{remark}
\newtheorem{remark}[theorem]{Remark}

\newtheorem{claim}[theorem]{Claim}
\newtheorem*{claim*}{Claim}
\newtheorem*{remark*}{Remark}
\newtheorem*{example*}{Example}
\newtheorem*{notation*}{Notation}

\numberwithin{equation}{section}

\def\R{{\mathbb R}}

\newcommand{\E}{{\mathbb E}}

\newcommand{\F}{{\mathcal F}}
\newcommand{\one}{{{\bf 1}}}

\renewcommand{\a}{\alpha}
\renewcommand{\b}{\beta}

\newcommand{\eps}{\varepsilon}
\renewcommand{\phi}{\varphi}

\newcommand{\bnu}{\boldsymbol{\nu}}
\newcommand{\bbeta}{\boldsymbol{\eta}}

\newcommand{\dd}{\mathrm{d}}
\newcommand{\dgrad}{\bar\nabla}

\providecommand{\abs}[1]{\left\lvert#1\right\rvert}
\providecommand{\norm}[1]{\left\Vert#1\right\Vert}

\newcommand{\cM}{\mathcal{M}}
\newcommand{\cH}{\mathcal{H}}
\newcommand{\cB}{\mathcal{B}}

\newcommand{\cW}{\mathcal{W}}
\newcommand{\cI}{\mathcal{I}}
\newcommand{\cL}{\mathcal{L}}

\newcommand{\cA}{\mathcal{A}}

\newcommand{\CE}{\mathcal{CE}}

\newcommand{\cF}{\mathcal{F}}

\newcommand{\cP}{\mathscr{P}}

\renewcommand{\tilde}{\widetilde}

\title[Gradient flows of the entropy for jump processes]{Gradient
  flows of the entropy for jump processes}

\author{Matthias Erbar}

\address{
University of Bonn\\
Institute for Applied Mathematics\\
Endenicher Allee 60\\
53115 Bonn\\
Germany}

\email{erbar@iam.uni-bonn.de}

\thanks{ME is supported by Bonn International Graduate School in
  Mathematics} 


\begin{document}

\begin{abstract}
  We introduce a new transportation distance between probability
  measures on $\R^d$ that is built from a L\'evy jump kernel. It is
  defined via a non-local variant of the Benamou-Brenier formula. We
  study geometric and topological properties of this distance, in
  particular we prove existence of geodesics. For translation
  invariant jump kernels we identify the semigroup generated by the
  associated non-local operator as the gradient flow of the relative
  entropy w.r.t. the new distance and show that the entropy is convex
  along geodesics.
\end{abstract}

\date\today

\maketitle

\tableofcontents

\section{Introduction}\label{sec:intro}

In the last two decades the theory of optimal transportation has found
applications to many areas of mathematics such as partial differential
equations, geometry and probability. We refer the reader to the
monograph \cite{Vil09} for an overview. In particular, optimal
transport has proved very useful in the study of diffusion
processes. One of the most striking examples is Otto's discovery
\cite{JKO98,O01} that many diffusion equations can be interpreted as
gradient flows of a suitable free energy functional with respect to
the $L^2$-Wasserstein distance on the space of probability measures. A
prominent example is the heat equation which is the gradient flow of
the Shannon entropy. By now, similar interpretations of the heat flow
have been established in a variety of settings ranging from Riemannian
manifolds to abstract metric measure spaces, see
\cite{Er10,OhSt09,FSS10,GKO10,AGS11a}.

The aim of this article is to build a bridge between the theory of
jump processes and non-local operators on one hand and ideas from
optimal transportation on the other hand. We will give a gradient flow
interpretation of the equation
\begin{align}\label{eq:general-evo}
  \partial_t u~=~\cL u\ ,
\end{align}
where $\cL$ is a non-local operator given by
\begin{align*}
 \cL u(x) ~=~\int u(y)-u(x)-(y-x)\cdot\nabla u(x)\one_{\{\abs{y-x}<1\}}J(x,\dd y)\ ,
\end{align*}
with a L\'evy measure $J(x,\dd y)$ for every $x\in\R^d$. Such
operators arise as the generators of a pure jump Feller process. For
this purpose the Wasserstein distance is not appropriate. The main
contribution of this article is thus the construction of a new
transportation distance on the space of probability measures that is
non-local in nature and allows to interpret equation
\eqref{eq:general-evo} formally as the gradient flow of the relative
entropy. We define this distance via a non-local variant of the
dynamical characterization of the Wasserstein distance by Benamou and
Brenier \cite{BB00}. A prominent example we will often consider is
given by the choice $J_\alpha(x,\dd
y)=c_\alpha\abs{y-x}^{-\alpha-d}\dd y$ with $\alpha\in(0,2)$
corresponding to the fractional Laplacian
$\cL=-(-\Delta)^{\frac{\alpha}{2}}$ which is a pseudo differential
operator with symbol $\abs{\xi}^\alpha$. For translation invariant
jump kernels such as $J_\alpha$ where the underlying jump process is a
L\'evy process, we rigorously identify the equation as the gradient
flow of the entropy w.r.t. the new distance in the framework of
gradient flows in metric spaces developed in \cite{AGS05}. Moreover,
we show that the entropy is convex along geodesics.

\medskip

To motivate our interest in such a link between jump processes and
optimal transport, let us highlight two observations.

The gradient flow approach has been used as a powerful tool in the
study of many evolution partial differential equations. Already in
Otto's original work \cite{O01} convexity properties of the entropy
functional have been used to derive explicit rates of convergence to
equilibrium for the porous medium equation. This approach is also well
adapted to the study of functional inequalities, such as logarithmic
Sobolev inequalities (see e.g.  the famous result by Otto-Villani
\cite{OV00}).  Recently, it has been shown that the gradient flow
characterization provides a good framework to study stability
properties of diffusion processes under changes of the driving
potential or the underlying geometry \cite{ASZ09}, \cite{Gi10}.

The regularity theory for elliptic and parabolic equations involving
non-local operators is under active development including both
analytic and probabilistic approaches (see e.g. \cite{CS11},
\cite{BBCK09} and references therein). In a local setting very precise
regularity results can be obtained using a lower bound on the Ricci
curvature of the operator in the sense of the Bakry-\'Emery criterion
\cite{BE85}. Equivalently, such curvature information can be encoded
into convexity properties of the entropy along Wasserstein
geodesics. In fact, geodesic convexity of the entropy has been used as
a synthetic notion of a lower Ricci curvature bound for metric measure
spaces by Lott--Villani \cite{LV09} and Sturm \cite{St06}. In this
sense the approach presented here could be used to define an
alternative notion of curvature in the spirit of Lott--Villani--Sturm
that might be more adapted to certain situations than the non-local
$\Gamma^2$-calculus. In the discrete setting of finite Markov chains,
this approach has already been used in \cite{EM11} to derive new
functional inequalities.

\medskip

Modifications of the Wasserstein distance have been considered
recently by a number of authors. In \cite{DNS09} Dolbeault, Nazaret
and Savar\'{e} proposed a new class of transport distances based on an
adaptation of the Benamou-Brenier formula to give a gradient flow
interpretation to a class of transport equations with non-linear
mobilities. Very recently, Maas \cite{Ma11} (see also \cite{Mie11b},
\cite{CHLZ11} for independent related work by Mielke and Chow et al.)
introduced a distance between probability measures on a discrete space
equipped with a Markov kernel such that the law of the continuous time
Markov chain evolves as the gradient flow of the entropy. Our approach
is very similar in spirit to the work of Maas and generalizes it to a
certain extend. On the technical side we use an adaptation of the
techniques developed in \cite{DNS09} to our non-local setting.

\subsection*{Main results}
Let us now discuss the content of this article in more detail.
Let $(J(x,\cdot),x\in\R^d)$ be a jump kernel. By this we mean that for
all $x\in\R^d$ $J(x,\cdot)$ is a Radon measure on $\R^d\setminus\{x\}$
depending measurably on $x$. Throughout this text $J$ shall satisfy
the following 

\begin{assumption}\label{ass:standing}
  For every bounded continuous function $f:\R^d\to\R$ the mapping
  \begin{align*}
    x~\mapsto~\int f(y) (1\wedge \abs{x-y}^2)J(x,\dd y)
  \end{align*}
  is again bounded and continuous.
\end{assumption}

In particular $(J(x,\cdot),x\in\R^d)$ is a so called L\'evy kernel
(see e.g. \cite[Ch. 3.5]{App04}). Further let $m$ be a Radon measure
on $\R^d$. We assume that $J$ is reversible w.r.t. $m$, i.e. the
measure $J(x,\dd y) m(\dd x)$ is symmetric.

We denote by $\cP(\R^d)$ the space of Borel probability measures on
$\R^d$. Given $\mu\in\cP(\R^d)$ we define its relative entropy
w.r.t. $m$ by
\begin{align*}
  \cH(\mu)~=~\int\rho\log\rho\; \dd m
  % \begin{cases}
  %   \int\rho\log\rho\; \dd m\ , & \mu=\rho m\ \text{and } (\rho\log\rho)_+\ \text{is integrable,}\\
  %   +\infty\ , & \text{else.} 
  % \end{cases}
\end{align*}
if $\mu$ is absolutely continuous w.r.t. $m$ with density $\rho$ and
$(\rho\log\rho)_+$ is integrable. Otherwise we set $\cH(\mu)=+\infty$.

\subsubsection*{A non-local transportation distance}
Let us first motivate the construction of our new metric by recalling
the dynamical characterization of the $L^2$-Wasserstein distance. The
Benamou-Brenier formula \cite{BB00} asserts that for two probability
densities $\bar\rho_0,\bar\rho_1$ on $\R^d$ we have
\begin{align}\label{eq:BBclassic}
  W^2_2(\bar\rho_0,\bar\rho_1)~=~\inf\limits_{\rho,\psi}\int_0^1\int\abs{\nabla\psi_t(x)}^2\rho_t(x)\dd
  x \dd t\ ,
\end{align}
where the infimum is taken over all sufficiently smooth functions
$\rho:[0,1]\times\R^d\to\R_+$ and $\psi:[0,1]\times\R^d\to\R$ subject
to the continuity equation
\begin{align}\label{eq:ceclassic}
  \begin{cases}
    \partial_t\rho + \nabla\cdot(\rho\nabla\psi)~=~0\ ,\\
    \rho_0=\bar\rho_0\ ,\ \rho_1=\bar\rho_1\ .
  \end{cases}
\end{align}
Here we will define a (pseudo-)metric (i.e. possibly attaining the
value $+\infty$) on $\cP(\R^d)$ by giving a non-local analogue of
formulas \eqref{eq:BBclassic} and \eqref{eq:ceclassic}.  In order to
obtain a metric with the desired properties it is necessary to
introduce a function $\theta:\R_+\times\R_+\to\R_+$ satisfying
Assumption \ref{ass:theta} below and to consider the mean
$\hat{\rho}(x,y):=\theta(\rho(x),\rho(y))$ of a given density
$\rho:\R^d\to \R$ at different points. We will be mostly interested in
the logarithmic mean
\begin{align}\label{eq:logmean-intro}
  \theta(s,t)=\frac{s-t}{\log s-\log t}
\end{align}
but for future use we allow for more generality in the
construction. For a function $\psi:\R^d\to\R$ we will denote by
$\dgrad\psi(x,y)=\psi(y)-\psi(x)$ its discrete gradient. Following the
approach of \cite{Ma11} one is led to consider the following
`distance'. Given probability measures $\bar\mu_0=\bar\rho_0 m$ and
$\bar\mu_1=\bar\rho_1m$ set
\begin{equation}\label{eq:metricnaiv}
  \tilde\cW(\bar\mu_0,\bar\mu_1)^2~:=~\inf\limits_{\rho,\psi}\frac{1}{2}\int_0^1\int\abs{\dgrad\psi_t(x,y)}^2\hat{\rho}_t(x,y)J(x,\dd y) m(\dd x)\dd t\ ,
\end{equation}
where the infimum is now taken over all functions $\rho$ and $\psi$
satisfying the `continuity equation'
\begin{align}\label{eq:cenaiv}
  \begin{cases}
    \partial_t\rho_t + \dgrad\cdot(\hat\rho_t\dgrad\psi_t)~=~0\ ,\\
   \rho_0=\bar\rho_0\ ,\ \rho_1=\bar\rho_1\ ,
  \end{cases}
\end{align}
in the sense that for every test function $\phi\in C^\infty_c(\R^d)$
we have
\begin{align*}
  \int\phi\partial_t\rho_t(x) m(\dd x) - \frac12\int\dgrad\phi(x,y)\dgrad\psi(x,y)\hat\rho(x,y)J(x,\dd y)m(\dd x)~=~0\ .
\end{align*}
Instead of addressing the variational problem \eqref{eq:metricnaiv}
directly we will adopt a measure theoretic point of view and recast it
in the more natural relaxed setting of time-dependent families of Radon
measures. Let us briefly sketch this approach. 

We let $G=\{(x,y)\in\R^d\times\R^d\ :\ x\neq y\}$ and fix $\gamma(\dd
x,\dd y)=J(x,\dd y)m(\dd x)$. We replace $\rho$ by a continuous curve
$t\mapsto\mu_t=\rho_tm$ in $\cP(\R^d)$ and $\psi_t$ induces a family
of signed Radon measures $\bnu_t(\dd x,\dd
y)=\dgrad\psi_t(x,y)\hat{\rho}_t(x,y)\gamma(\dd x,\dd y)$ on $G$. The
couple $(\mu,\bnu)$ now satisfies the linear equation
\begin{align}\label{eq:ce-intro}
  \begin{cases}
    \partial_t\mu_t+\dgrad\cdot\bnu_t~=~0\ ,\\
    \mu_0=\bar\mu_0,\ \mu_1=\bar\mu_1
  \end{cases}
\end{align}
which we understand in the sense of distributions, i.e. for all test
functions $\phi\in C^\infty_c((0,1)\times\R^d)$ :
\begin{align*}
  \int_0^1\int\partial_t\phi\dd\mu_t\dd t + \frac12\int_0^1\int\dgrad\phi(x,y)\bnu_t(\dd x,\dd y)\dd t~=~0\ .
\end{align*}
The quantity to be minimized in \eqref{eq:metricnaiv} can now be
rewritten as
\begin{align*}
\frac12\int_0^1\int\abs{\frac{\dd\bnu_t}{\dd\gamma}(x,y)}^2\theta\left(\frac{\dd\mu_t}{\dd m}(x),\frac{\dd\mu_t}{\dd m}(y)\right)^{-1}\gamma(\dd x,\dd y)\dd t\ .
\end{align*}
We will define a distance $\cW$ by proceeding as follows. To any
$\mu\in\cP(\R^d)$ we associate two Radon measures on $G$ by setting
$\mu^1(\dd x,\dd y)=J(x,\dd y)\mu(\dd x)$ and $\mu^2(\dd x,\dd y)=J(y,\dd
x)\mu(\dd y)$. Given a Radon measure $\bnu$ on $G$ we choose a
reference measure $\sigma$ on $G$ such that $\bnu=w\sigma$ and
$\mu^i=\rho^i\sigma,\ i=1,2$ are all absolutely continuous
w.r.t. $\sigma$. Then we define the action functional by
\begin{align*}
  \cA(\mu,\bnu)~:=~\frac12\int\abs{\frac{\dd\bnu}{\dd\sigma}}^2\theta\left(\frac{\dd\mu^1}{\dd\sigma},\frac{\dd\mu^2}{\dd\sigma}\right)^{-1}\; \dd \sigma\ .
\end{align*}
Assumptions on $\theta$ will guarantee that the map $(w,s,t)\mapsto
w^2\theta(s,t)^{-1}$ is homogeneous, hence the definition of $\cA$ is
independent of the choice of $\sigma$. Given two measures
$\bar\mu_0,\bar\mu_1\in\cP(\R^d)$ we denote by
$\CE_{0,1}(\bar\mu_0,\bar\mu_1)$ the set of all sufficiently regular
solutions (to be made precise in section \ref{sec:nonlocal-ce})
$(\mu_t,\bnu_t)_{t\in[0,1]}$ of the continuity equation
\eqref{eq:ce-intro}.

\begin{definition*}\label{def:metric-intro}
 For $\bar\mu_0,\bar\mu_1\in\cP(\R^d)$ we define
\begin{align*}%\label{eq:defmetric-intro}
\cW(\bar\mu_0,\bar\mu_1)^2~:=~\inf\left\{\int_0^1\cA(\mu_t,\bnu_t)\dd t\ :\quad
(\mu,\bnu)\in\CE_{0,1}(\bar\mu_0,\bar\mu_1)\right\}\ .
\end{align*}
\end{definition*}

It is unclear whether $\cW$ coincides with $\tilde\cW$ defined in
\eqref{eq:metricnaiv} in full generality. However, we will give a
positive answer for the more restricted case of a sufficiently regular
translation invariant jump kernel such as $J_\a$ (see Proposition
\ref{prop:equiv-definition}). We can now state the first main result
of this article.

\begin{theorem}\label{thm:1-metric}
  $\cW$ defines a (pseudo-) metric on $\cP(\R^d)$ . The topology it
  induces is stronger than the topology of weak convergence. For each
  $\tau\in\cP(\R^d)$ the set $\cP_\tau:=\{\mu\in\cP(\R^d)\ :\
  \cW(\mu,\tau)<\infty\}$ equipped with the distance $\cW$ is a
  complete geodesic space.
\end{theorem}

\subsubsection*{Gradient flow of the entropy}

Let us give a short formal argument why equation
\eqref{eq:general-evo} can be seen as the gradient flow of the
relative entropy w.r.t. the distance $\cW$ if we choose $\theta$ to be
the logarithmic mean.

In the classical setting many partial differential equations of the form 
\begin{align*}
  \partial_t\rho - \nabla\cdot\big(\rho\nabla f'(\rho)\big)~=~0
\end{align*}
can, at least formally, be seen as the gradient flow of the integral
functional $\cF(\rho)=\int f(\rho)\dd m$ w.r.t. the $L^2$-Wasserstein
distance. Hence in the new geometry determined by the distance
$\tilde\cW$ via \eqref{eq:metricnaiv}, \eqref{eq:cenaiv} the gradient
flow of the functional $\cF$ should be given by the equation
\begin{align*}
  \partial_t\rho - \dgrad\cdot\big(\hat\rho\dgrad f'(\rho)\big)~=~0\ .
\end{align*}
If we now consider the relative entropy $\cH$ we have $f'(r)=1 + \log
r$. Taking into account \eqref{eq:logmean-intro} we see that the
corresponding gradient flow is given by
\begin{align*}
\partial_t\rho - \dgrad\cdot\big(\dgrad \rho\big)~=~0\ ,
\end{align*}
which is a weak formulation of \eqref{eq:general-evo}. In particular
we see that the appearance of the logarithmic mean is necessary in
order to account for the fact that the discrete gradient lacks a chain
rule.

In the more restricted setting of a translation invariant jump kernel
we can indeed rigorously identify equation \eqref{eq:general-evo} as
the gradient flow of the relative entropy w.r.t. the corresponding
metric $\cW$ in the framework of the metric theory developed in
\cite{AGS05}. So assume for the rest of this introduction that $J$ satisfies
\begin{align*}
  J(x+z,A+z)~=~J(x,A) \qquad\forall x,z\in\R^d,\ A\subset\R^d\setminus\{x\}
\end{align*}
and let $m$ be Lebesgue measure. Then we can write $J(x,A)=\nu(A-x)$
for a L\'evy measure $\nu$ on $\R^d\setminus\{0\}$. The operator $\cL$
generates a semigroup $P_t=\exp(t\cL)$ in $L^2(\R^d)$ that can be
represented by kernel $p_t$:
\begin{align*}
  P_tf(x)~=~\int f(y)p_t(x,\dd y) \ . 
\end{align*}
In fact $p_t$ is the transition kernel of the L\'evy process with
characteristic triplet $(0,0,\nu)$ in the sense of the
L\'evy-Khinchine formula (see e.g. \cite{App04}). In the same way
$\cL$ generates a semigroup on $\cP(\R^d)$. Under certain further
regularity assumptions on the transition kernel (see Section
\ref{sec:gradflow} for a precise statement) we prove the following

\begin{theorem}\label{thm:2-gradflow}
  The semigroup $P$ generated by $\cL$ is the gradient flow of the
  relative entropy in the sense that it satisfies the Evolution
  Variational Inequality (EVI): For any $\mu\in\cP(\R^d)$ and
  $\sigma\in\cP_\mu$ we have
  \begin{align}\label{eq:EVI-intro}
    \frac12\frac{\dd^+}{\dd t}\cW^2(P_t\mu,\sigma) + \cH(P_t\mu)~\leq~\cH(\sigma)\quad\forall t>0\ .
  \end{align}
  Moreover the entropy is convex along $\cW$-geodesics. More
  precisely, let $\mu_0,\mu_1\in\cP(\R^d)$ such that
  $\cW(\mu_0,\mu_1)<\infty$ and let $(\mu_t)_{t\in[0,1]}$ be a
  geodesic connecting $\mu_0$ and $\mu_1$. Then we have
  \begin{align*}
    \cH(\mu_t)~\leq~(1-t)\cH(\mu_0) + t \cH(\mu_1)\ .
  \end{align*}
\end{theorem}

Among several ways to characterize gradient flows in metric spaces,
the EVI is one of the strongest. For example it implies geodesic
convexity of the entropy (see \cite{DS08}). Convexity of the entropy
along $\cW$-geodesics can be seen as a non-local analogue of McCann's
displacement convexity \cite{McC97}, which corresponds to convexity
along geodesics of the $L^2$-Wasserstein distance. For the choice
$\nu(\dd y)=c_\alpha\abs{y}^{-\alpha-d}\dd y$ with $\alpha\in(0,2)$
and a suitable constant $c_\a$ we obtain the following

\begin{corollary}\label{cor:frac-laplace-intro}
  The semigroup generated by the fractional Laplacian
  $-(-\Delta)^{\frac{\alpha}{2}}$ is the gradient flow of the relative
  entropy w.r.t. the metric $\cW$ built from the jump kernel
  $J_\a(x,dy)=c_\alpha\abs{y-x}^{-\alpha-d}\dd y$.
\end{corollary}

We expect that a similar result should also hold for semigroups
associated to suitable non-homogeneous jump kernels $J$. It would be
desirable to find examples of kernels where the entropy is strictly
geodesically convex. This could be exploited to derive new functional
inequalities and rates of convergence to equilibrium for the
corresponding evolution equation, as has been done in the discrete
setting of finite Markov chains in \cite{EM11}. However, establishing
a stronger EVI($\kappa$) in concrete examples does not seem to be an
easy task and we will address this question in a forthcoming
publication. Moreover, we expect that the approach presented here can
be generalized in order to give a gradient flow interpretation to
evolution equations associated to L\'evy-type operators with both
non-local and diffusion part.

\subsection*{Organization of the paper}
In Section \ref{sec:action} we study the action functional $\cA$ and
establish various properties needed in the sequel. Section
\ref{sec:nonlocal-ce} is devoted to an analysis of the non-local
continuity equation \eqref{eq:ce-intro}. In Section \ref{sec:metric}
we define the metric $\cW$ and prove Theorem
\ref{thm:1-metric}. Finally, we focus on translation invariant jump
kernels and present the proof of Theorem \ref{thm:2-gradflow} in
Section \ref{sec:gradflow}.

 \subsection*{Acknowledgment} {\small The author is grateful to Jan
   Maas and Karl-Theodor Sturm for stimulating discussions on this
   paper and related questions.}

\section{The action functional}\label{sec:action}

In this section we introduce and study an action functional on pairs
of measures. Let us first introduce some notation. We denote by
$\cP(\R^d)$ the space of Borel probability measures on $\R^d$ equipped
with the topology of weak convergence. We let
$G=\{(x,y)\in\R^d\times\R^d\vert x\neq y\}$ and denote by
$\cM_{loc}(G)$ the space of signed Radon measures on the open set $G$
equipped with the weak* topology in duality with continuous functions
with compact support in $G$.

The definition of the action functional and later the metric will
depend on the choice of a function $\theta:\R_+\times\R_+\to\R_+$. We
will always require it to fulfill the following assumptions:
\begin{assumption} \label{ass:theta} The function $\theta$ has the
  following properties:
\begin{itemize}
\item[(A1)] (Regularity): $\theta$ is continuous on $\R_+ \times \R_+$
  and $C^1$ on $(0,\infty) \times (0,\infty)$;
\item[(A2)] (Symmetry): $\theta(s,t) = \theta(t,s)$ for $s, t \geq 0$;
\item[(A3)] (Positivity, normalisation): $\theta(s,t) > 0$ for $s,t >
  0$ and $\theta(1,1)=1$;
\item[(A4)] (Zero at the boundary): $\theta(0,t) = 0$ for all $t \geq
  0$;
\item[(A5)] (Monotonicity): $\theta(r, t) \leq \theta(s,t)$ for all $0
  \leq r \leq s$ and $t \geq 0$;
\item[(A6)] (Positive homogeneity): $\theta(\lambda s, \lambda t) =
  \lambda \theta(s,t)$ for $\lambda > 0$ and $s,t \geq 0$;
\item[(A7)] (Concavity): the function $\theta : \R_+ \times \R_+ \to
  \R_+$ is concave.
\end{itemize}
\end{assumption}
It is easy to check that these assumptions imply
\begin{equation}\label{eq:theta-arithmetic-ineq}
\theta(s,t)~\leq~\frac{s+t}{2}\quad\forall s,t\geq0\ .
\end{equation}

In view of applications to gradient flows of the entropy we will be
mostly interested in a particular choice of $\theta$, namely the
logarithmic mean given by
\begin{align}\label{eq:log-mean}
  \theta(s,t)~=~\int_0^1s^\alpha t^{1-\alpha}\dd \alpha~=~\frac{s-t}{\log s-\log t}\ ,
\end{align}
the latter expression being valid for $s,t>0$. However, for future use
we will allow for more generality in the choice of $\theta$. Given a
function $\rho:\R^d\to\R_+$ we will often write
\begin{align*}
  \hat\rho(x,y)~:=~\theta(\rho(x),\rho(y))\ .
\end{align*}
We can now define a function $\alpha : \R\times\R_+\times\R_+ \to \R_+\cup\{\infty\}$, called the action density function, by setting
\begin{align*}
 \alpha(w,s,t) :=  \left\{ \begin{array}{ll}
  \frac{w^2}{2\theta(s,t)}\;,
  & \theta(s,t) \neq 0\;,\\
0\;,
  &  \theta(s,t) = 0\text{ and } w = 0 \;,\\ 
+ \infty\;,
  & \theta(s,t) = 0 \text{ and } w \neq 0\;.\end{array} \right.
\end{align*}
The following observation will be useful.

\begin{lemma}\label{lem:actiondensity}
  The function $\alpha$ is lower semicontinuous, convex and
  positively homo\-gen\-eous, i.e.
  \begin{align*}
    \alpha(\lambda w,\lambda s,\lambda t)~=~\lambda \alpha(w,s,t)\quad
    \forall w\in\R\;,\ s,t\geq0\;,\ \lambda\geq 0\ .
  \end{align*}
\end{lemma}
\begin{proof}
  This is easily checked using (A6),(A7) and the convexity of the
  function $(x,y)\mapsto \frac{x^2}{y}$ on $\R \times (0,\infty)$.
\end{proof}

We will now define an action functional on pairs of measures
$(\mu,\bnu)$ where $\mu\in\cP(\R^d)$ and $\bnu\in\cM_{loc}(G)$. To
$\mu$ we associate a two Radon measures in $\cM_{loc}(G)$ by setting:
\begin{align}\label{eq:defmu12}
  \mu^1(\dd x,\dd y)~:=~J(x,\dd y)\mu(\dd x)\;,\ \mu^2(\dd x,\dd
  y)~:=~J(y,\dd x)\mu(\dd y)\ .
\end{align}
We can always choose a measure $\sigma\in\cM_{loc}(G)$ such that
$\mu^i=\rho^i\sigma,\ i=1,2$ and $\bnu=w\sigma$ are all absolutely
continuous with respect to $\sigma$. For example take the sum of the
total variations
 $\sigma:=\abs{\mu^1}+\abs{\mu^2}+\abs{\bnu}$.  We can
then define the \emph{action functional} by
\begin{align*}
\cA(\mu,\bnu)~:=~\int\alpha\big(w,\rho^1,\rho^2\big)\dd\sigma\ .
\end{align*}
Note that this definition is independent of the choice of $\sigma$
since $\a$ is positively homogeneous. Hence we can also write the
action functional as
\begin{align*}
  \cA(\mu,\bnu)~=~\int\alpha\left(\frac{\dd\lambda_1}{\dd\abs{\lambda}},\frac{\dd\lambda_2}{\dd\abs{\lambda}},\frac{\dd\lambda_3}{\dd\abs{\lambda}}\right)\dd\abs{\lambda}\ ,
\end{align*}
where $\lambda$ is the vector valued measure given by
$\lambda=(\bnu,\mu^1,\mu^2)$. 

In the case where the measure $\mu$ is absolutely continuous
w.r.t. $m$ the next lemma shows that the action takes a more intuitive
form. For this we denote by $Jm\in\cM_{loc}(G)$ the measure given by
$Jm(\dd x,\dd y)=J(x,\dd y)m(\dd x)$.

\begin{lemma}\label{lem:densities}
  Let $\mu\in\cP(\R^d)$ be absolutely continuous w.r.t. $m$ with
  density $\rho$. Further let $\bnu\in\cM_{loc}(G)$ such that
  $\cA(\mu,\bnu)<\infty$. Then there exist a function $w:G\to\R$ such
  that $\bnu=w\hat\rho Jm$ and we have
  \begin{align}\label{eq:densities}
    \cA(\mu,\bnu)~=~\frac12\int\abs{w(x,y)}^2\hat\rho(x,y)J(x,\dd y)m(\dd x)\ .
  \end{align}
\end{lemma}

\begin{proof}
  Choose $\lambda\in\cM_{loc}(G)$ such that $Jm=h\lambda$ and
  $\bnu=\tilde w\lambda$ are both absolutely continuous
  w.r.t. $\lambda$. Note that $\mu^i=\rho^iJm,\ i=1,2$ with
  $\rho^1(x,y)=\rho(x)$ and $\rho^2(x,y)=\rho(y)$. Further, we denote
  by $\tilde\rho^i$ the density of $\mu^i$ w.r.t $\lambda$. Now by
  definition,
  \begin{align}\label{eq:densities1} 
   \cA(\mu,\bnu)~=~\int\a\big(\tilde w,\tilde\rho^1,\tilde\rho^2\big)\dd\lambda~<~\infty\ .
  \end{align}
  Let $A\subset G$ such that $\int_A\theta(\rho^1,\rho^2)\dd Jm=0$. From the
  homogeneity of $\theta$ we conclude
  \begin{align*}
   0~=~\int_A\theta(\rho^1,\rho^2)\dd Jm~=~\int_A\theta(\tilde\rho^1,\tilde\rho^2)\dd\lambda\ ,
  \end{align*}
  i.e. $\theta(\tilde\rho^1,\tilde\rho^2)=0$ $\lambda$-a.e. on
  $A$. Now the finiteness of the integral in \eqref{eq:densities1}
  implies that $\tilde w=0$ $\lambda$-a.e. on $A$. In other words
  $\bnu(A)=0$ and hence $\bnu$ is absolutely continuous w.r.t. the
  measure $\hat\rho Jm$. Formula \eqref{eq:densities} now follows
  immediately from the homogeneity of $\a$.
\end{proof}

\begin{lemma}[Lower semicontinuity of the action]\label{lem:lscaction}
  $\cA$ is lower semicontinuous w.r.t. weak convergence of
  measures. More precisely, assume that $\mu_n\rightharpoonup\mu$
  weakly in $\cP(\R^d)$ and $\bnu_n\rightharpoonup^*\bnu$ weakly* in
  $\cM_{loc}(G)$. Then
  \begin{align*}
    \cA(\mu,\bnu)~\leq~\liminf\limits_{n}\cA(\mu_n,\bnu_n)\ .
  \end{align*}
\end{lemma}

\begin{proof}
  Note that by Assumption \ref{ass:standing} the weak convergence of
  $\mu_n$ to $\mu$ implies the weak* convergence of $\mu^i_n$ to
  $\mu^i$ in $\cM_+(G)$ for $i=1,2$. Now the claim follows immediately
  from a general result on integral functionals, Proposition
  \ref{prop:lsc}.
\end{proof}

\begin{proposition}[{\cite[Thm. 3.4.3]{Bu89}}]\label{prop:lsc}
  Let $\Omega$ be a locally compact Polish space and let
  $f:\Omega\times\R^n\to[0,+\infty]$ be a lower semicontinuous
  function such that $f(\omega,\cdot)$ is convex and positively
  $1$-homogeneous for every $\omega\in\Omega$. Then the functional
  \begin{align*}
    F(\lambda)~=~\int\limits_\Omega
    f\left(\omega,\frac{\dd\lambda}{\dd\abs{\lambda}}(\omega) \right)\abs{\lambda}(\dd \omega)
  \end{align*}
  is sequentially weak* lower semicontinuous on the space of vector
  valued signed Radon measures $\cM_{loc}(\Omega,\R^n)$.
\end{proposition}

The next estimate will be crucial for establishing compactness of
families of curves with bounded action in Section
\ref{sec:nonlocal-ce}.

\begin{lemma}\label{lem:integrability}
  \begin{itemize}
    \item[i)] 
    There exists a constant $C>0$ such that for all $\mu\in\cP(\R^d)$
    and $\bnu\in\cM_{loc}(G)$ we have:
    \begin{align*}
      \int\limits_{G}\big(1\wedge\abs{x-y}\big)\abs{\bnu}(\dd x,\dd y)~\leq~C\sqrt{\cA(\mu,\bnu)}\;.
    \end{align*}
  \item[ii)] For each compact set $K\subset G$ there exists a
    constant $C(K)>0$ such that for all $\mu\in\cP(\R^d)$ and
    $\bnu\in\cM_{loc}(G)$ we have:
    \begin{align*}
      \abs{\bnu}(K)~\leq~C(K)\sqrt{\cA(\mu,\bnu)}\;.
    \end{align*}
  \end{itemize}
\end{lemma}

\begin{proof}
  To prove i) let us define the measure
  $\lambda=\abs{\mu^1}+\abs{\mu^2}+\abs{\bnu}$ and write
  $\mu^i=\rho^i\lambda,\ \bnu=w\lambda$. We can assume
  that $\cA(\mu,\bnu)<\infty$ as otherwise there is nothing to
  prove. This implies that the set $A=\{(x,y)\ \vert\
  \a(w,\rho^1,\rho^2)=\infty\}$ has zero measure with respect to
  $\lambda$. We can now estimate:
\begin{equation*}
\begin{split}
  &\qquad\int\limits_{G}\big(1\wedge\abs{x-y}\big)\abs{\bnu}(\dd x,\dd y)\\~&\leq~\int\limits_{G}\big(1\wedge\abs{x-y}\big)\abs{w}\dd\lambda\\
  &=~\int\limits_{A^c}\big(1\wedge\abs{x-y}\big)\sqrt{2\theta(\rho^1,\rho^2)}\sqrt{\a(w,\rho^1,\rho^2)}\dd\lambda\\
  &\leq~\left(\int\limits_{G}\big(1\wedge\abs{x-y}^2\big)2\theta(\rho^1,\rho^2)\dd\lambda\right)^{\frac{1}{2}}\left(\int\limits_{G}\a(w,\rho^1,\rho^2)\dd\lambda\right)^{\frac{1}{2}}\\
  &\leq~C\sqrt{\cA(\mu,\bnu)}\ .
 \end{split}  
\end{equation*}
The last inequality follows, since by the estimate
\eqref{eq:theta-arithmetic-ineq} and Assumption
\ref{ass:standing} we have :
\begin{align*}
  \int\limits_{G}\big(1\wedge\abs{x-y}^2\big)\theta(\rho^1,\rho^2)\dd\lambda~&\leq~\int\limits_{G}\big(1\wedge\abs{x-y}^2\big)\frac12(\rho^1+\rho^2)\dd\lambda\\
  &=~\int\limits_{G}\big(1\wedge\abs{x-y}^2\big)J(x,\dd y)\mu(\dd x)\\
  &\leq~\sup\limits_x\int(1\wedge\abs{x-y}^2)J(x,\dd y)~<~\infty\ .
\end{align*}
To prove ii) we note that by a similar argument
\begin{align*}
  \abs{\bnu}(K)~\leq~\left(\int\limits_K2J(x,\dd y)\mu(\dd x)\right)^\frac12 \sqrt{\cA(\mu,\bnu)}\ .
\end{align*}
\end{proof}

\begin{lemma}[Convexity of the action]\label{lem:convexityPhi}
  Let $\mu^j\in\cP(\R^d)$ and $\bnu^j\in\cM_{loc}(G)$ for $j=0,1$. For
  $\tau\in[0,1]$ set $\mu^\tau=\tau\mu^1+(1-\tau)\mu^0$ and
  $\bnu^\tau=\tau\bnu^1+(1-\tau)\bnu^0$. Then we have :
$$\cA(\mu^\tau,\bnu^\tau)~\leq~\tau\cA(\mu^1,\bnu^1)+(1-\tau)\cA(\mu^0,\bnu^0)\ .$$
\end{lemma}

\begin{proof}
  Let us fix a reference measure $\lambda\in\cM_{loc}(G)$ such that
  $\mu^{j,i},\bnu^j$ for $j=0,1$ and $i=1,2$ are all absolutely
  continuous w.r.t. $\lambda$ and write $\mu^{j,i}=\rho^{j,i}\lambda$
  and $\bnu^j=w^j\lambda$. Note that
  $\mu^{\tau,i}=\rho^{\tau,i}\lambda$ with
  $\rho^{\tau,i}=\tau\rho^{1,i}+(1-\tau)\rho^{0,i}$ and
  $\bnu^\tau=w^\tau\lambda$ with $w^\tau=\tau w^1+(1-\tau)w^0$. From
  the convexity of the action density function $\a$ we obtain :
\begin{equation*}
 \begin{split}
   \cA(\mu^\tau,\bnu^\tau)~&=~\int\a(w^\tau,\rho^{\tau,1},\rho^{\tau,2})\dd\lambda\\
   &\leq~\tau \int\a(w^1,\rho^{1,1},\rho^{1,2})\dd\lambda+(1-\tau)\int\a(w^0,\rho^{0,1},\rho^{0,2})\dd\lambda\\
   &=~\tau \cA(\mu^1,\bnu^1)+ (1-\tau)\cA(\mu^0,\bnu^0)\ .
 \end{split}
\end{equation*}
\end{proof}

We will now show that the action functional enjoys a monotonicity
property under convolution if we assume that the jump kernel is
translation invariant in the sense that
\begin{align}\label{eq:translationinv-J}
  J(x-z,A-z)~=~J(x,A)\qquad \forall x,z\in\R^d, A\in\cB(\R^d)\ .
\end{align}
For the rest of this section we also assume that $m$ is Lebesgue
measure. We first need to fix a way of convoluting measure on $\R^d$
and on $G$ in a consistent manner. Let $k$ be a convolution kernel,
i.e. $k:\R^d\to\R_+$ satisfying $\int k(z)\dd z=1$. Given a measure
$\mu\in\cP(\R^d)$, its convolution is defined as usual by
\begin{align*}
  (\mu*k)(A)~:=~\int k(z)\mu(A-z)\dd z\qquad \forall A\in\cB(\R^d)\ .
\end{align*}
On the other hand given a measure $\bnu\in\cM_{loc}(G)$ we define
$\bnu*k\in\cM_{loc}(G)$ by setting for all Borel measurable sets
$B\subset G$
\begin{align}\label{eq:convolutiondiagonal}
  (\bnu*k)(B)~:=~\int k(z)\bnu(B-\binom{z}{z})\dd z\ .
\end{align}
Note that this implies in particular that for every bounded function
$f:G\to \R$ with compact support in $G$ we have:
\begin{align*}
  \int f(x,y) (\bnu*k)(\dd x,\dd y)~=~\int\int k(z)f(x+z,y+z)\bnu(\dd x,\
dd y)\dd z\ .
\end{align*}
We now have the following monotonicity property under convolution.
\begin{proposition}\label{prop:monotonconvolution}
  Assume that $J$ satisfies \eqref{eq:translationinv-J} and let $k$
  be a convolution kernel. Then for every
  $\mu\in\cP(\R^d),\bnu\in\cM_{loc}(G)$ we have
  \begin{align}\label{eq:monotonconvolution}
    \cA( \mu*k,\bnu*k )~\leq~\cA( \mu,\bnu )\ .
  \end{align}
\end{proposition}
\begin{proof}
  We can assume without restriction that $\cA(\mu,\bnu)$ is finite as
  otherwise there is nothing to proof. Let us introduce the maps
  $\tau_z:x\mapsto x+z$ for $z\in\R^d$ and let us denote by
  $\mu_z,\bnu_z$ the push forward $(\tau_z)_*\mu=\mu(\cdot-z)$,
  resp. $(\tau_z\times\tau_z)_*\bnu=\bnu(\cdot-\binom{z}{z})$. Using
  the convexity of the action functional, Lemma
  \ref{lem:convexityPhi}, together with its lower semicontinuity,
  Lemma \ref{lem:lscaction}, we see that
  \begin{align*}
    \cA( \mu*k,\bnu*k )~\leq~\int\cA( \mu_z,\bnu_z ) k(z)\dd z\ .
  \end{align*}
  Thus the proof is complete if we show that $\cA( \mu_z,\bnu_z)=\cA(
  \mu,\bnu )$ for all $z\in\R^d$. To this end recall the definition
  \eqref{eq:defmu12}. Using the the invariance property
  \eqref{eq:translationinv-J} it is immediate to check that
  $\mu_z^i=(\tau_z\times\tau_z)_*\mu^i$ for $i=1,2$. Now choose
  $\lambda\in\cM_{loc}(G)$ with $\mu^i=\rho^i\lambda$ and
  $\bnu=w\lambda$. Then for all $z\in\R^d$ we have
  $(\mu_z)^i=(\mu^i)_z=\rho^i(\cdot-\binom{z}{z})\lambda_z$ and
  $\bnu_z=w(\cdot-\binom{z}{z})\lambda_z$. Hence we finally obtain
  \begin{align*}
    \cA( \mu_z,\bnu_z )~&=~\int\a\left(w(\cdot-\binom{z}{z}),\rho^1(\cdot-\binom{z}{z}),\rho^2(\cdot-\binom{z}{z})\right)\dd\lambda_z\\
    ~&=~\int\a(w,\rho^1,\rho^2)\dd\lambda~=~\cA(\mu,\bnu)\ .
  \end{align*}
\end{proof}

\section{A non-local continuity equation}\label{sec:nonlocal-ce}

In this section we will consider the continuity equation
\begin{equation}\label{eq:ce}
 \partial_t\mu_t + \dgrad\cdot\bnu_t~=~0 \qquad \mbox{on }(0,T)\times\R^d\ .
\end{equation}
Here $(\mu_t)_{t\in[0,T]}$ and $(\bnu_t)_{t\in[0,T]}$ are Borel
families of measures in $\cP(\R^d)$ and $\cM_{loc}(G)$ respectively
such that
\begin{equation}\label{eq:cefinitemass}
\int_0^T\int\big(1\wedge\abs{x-y}\big)\abs{\bnu_t}(\dd x,\dd y)\dd t~<~\infty\ .
\end{equation}
We suppose that \eqref{eq:ce} holds in the sense of
distributions. More precisely, we require that for all $\varphi\in
C^\infty_c((0,T)\times\R^d)$ :
\begin{equation}\label{eq:cedistribution}
  \int_0^T\int\partial_t\varphi_t(x) \mu_t(\dd x) \dd t +
  \frac12\int_0^T\int\dgrad\varphi_t(x,y) \bnu_t(\dd x,\dd y)\dd t~=~0\ .
\end{equation}
Recall that for a function $\varphi:\R^d\to\R$ we denote by
$\dgrad\varphi(x,y)=\varphi(y)-\varphi(x)$ the discrete gradient. Note
that \eqref{eq:cefinitemass} is a natural integrability assumption one
should make to ensure that the second term in
\eqref{eq:cedistribution} is well-defined. The following is an
adaptation of \cite[Lemma 8.1.2]{AGS05}.

\begin{lemma}\label{lem:cont-represent}
  Let $(\mu_t)_{t\in[0,T]}$ and $(\bnu_t)_{t\in[0,T]}$ be Borel
  families of measures in $\cP(\R^d)$ and $\cM_{loc}(G)$ satisfying
  \eqref{eq:ce} and \eqref{eq:cefinitemass}. Then there exists a
  weakly continuous curve $(\tilde\mu_t)_{t\in[0,T]}$ such that
  $\tilde\mu_t=\mu_t$ for a.e. $t\in[0,T]$. Moreover, for every
  $\varphi\in C^\infty_c([0,T]\times\R^d)$ and all $0\leq t_0\leq
  t_1\leq T$ we have :
  \begin{equation}\label{eq:cedistributionrefined}
    \int\varphi_{t_1} \dd\tilde\mu_{t_1}-\int\varphi_{t_0} \dd\tilde\mu_{t_0}~=~\int_{t_0}^{t_1}\int\partial_t\varphi \dd\mu_t \dd t + \frac12\int_{t_0}^{t_1}\int\dgrad\varphi \dd\bnu_t\dd t\ .
  \end{equation}
 \end{lemma}

\begin{proof}
  Let us set
  \begin{align*}
    V(t)~:=~\int\big(1\wedge \abs{x-y}\big)\abs{\bnu_t}(\dd x,\dd y)\ .
  \end{align*}
  By assumption $t\mapsto V(t)$ belongs to $L^1(0,T)$. Fix $\xi\in
  C^\infty_c(\R^d)$. We claim that the map
  $t\mapsto\mu_t(\xi)=\int\xi\dd\mu_t$ belongs to
  $W^{1,1}(0,T)$. Indeed, using test functions of the form
  $\varphi(t,x)=\eta(t)\xi(x)$ with $\eta\in C^\infty_c(0,T)$,
  equation \eqref{eq:cedistribution} shows that the distributional
  derivative of $\mu_t(\xi)$ is given by
  \begin{align*}
    \dot{\mu}_t(\xi)~=~\frac12\int\dgrad\xi\dd\bnu_t
  \end{align*}
  for a.e. $t\in(0,T)$ and we can estimate 
  \begin{align}\label{eq:weakderiv}
   \abs{\dot{\mu}_t(\xi)}~\leq~\frac12\int\abs{\dgrad\xi}\dd\abs{\bnu_t}~\leq~\frac12\norm{\xi}_{C^1}V(t)\ .
  \end{align}
  Based on \eqref{eq:weakderiv} we can argue as in \cite[Lemma
  8.1.2]{AGS05} to obtain existence of a weakly continuous
  representative $t\mapsto\tilde\mu_t$.

  To prove \eqref{eq:cedistributionrefined} fix $\varphi\in
  C^\infty_c([0,T]\times\R^d)$ and choose $\eta_\varepsilon\in
  C^\infty_c(t_0,t_1)$ such that
  \begin{align*}
    0\leq\eta_\varepsilon\leq 1\ ,\quad
    \lim\limits_{\varepsilon\to0}\eta_\varepsilon(t)=1_{(t_0,t_1)}(t)\
    \forall t\in[0,T]\ ,\quad
    \lim\limits_{\varepsilon\to0}\eta_\varepsilon'=\delta_{t_0}-\delta_{t_1}\;.
  \end{align*}
  Now equation \eqref{eq:cedistribution} implies
  \begin{align*}
    -\int_0^T\eta_\varepsilon'\int\varphi
    \dd\tilde\mu_t\dd t~=~\int_0^T\eta_\varepsilon\int\partial_t\varphi
    \dd\mu_t \dd t +
    \frac12\int_0^T\eta_\varepsilon\int\dgrad\varphi
    \dd\bnu_t \dd t\ .
  \end{align*}
  Thanks to the continuity of $t\mapsto\tilde\mu_t$ we can pass to
  limit as $\varepsilon\to0$ and obtain
  \eqref{eq:cedistributionrefined}.
\end{proof}
In view of the previous Lemma it makes sense to define solutions to
the continuity equation in the following way.
\begin{definition}\label{def:ce}
  We denote by $\CE_{T}(\bar\mu_0,\bar\mu_1)$ the set of all pairs
  $(\mu,\bnu)$ satisfying the following conditions:
  \begin{align} \label{eq:ce-conditions} \left\{ \begin{array}{ll}
        {(i)} & \mu : [0,T] \to \cP(\R^d)  \text{ is weakly continuous}\;;\\
        {(ii)} &  \mu_0 = \bar\mu_0\;, \quad \mu_T = \bar\mu_1\;; \\
        {(iii)} & (\bnu_t)_{t\in[0,T]}  \text{ is a Borel family of measures in } \cM_{loc}(G)\;;\\
        {(iv)} & \int_0^T\int\big(1\wedge \abs{x-y}\big)\abs{\bnu_t}(\dd x,\dd y)\dd t~<~\infty\;;\\
        {(v)} &  \text{We have in the sense of distributions:}\\
        &\displaystyle{\partial_t\mu_t + \dgrad\cdot\bnu_t~=~0}\;.\
      \end{array} \right.
  \end{align}
\end{definition}
The following result will allow us to extract subsequential limits
from sequences of solutions to the continuity equation which have
bounded action.
\begin{proposition}[Compactness of solutions to the continuity
  equation]\label{prop:cecompactness}
  Let $(\mu^n,\bnu^n)$ be a sequence in
  $\CE_{T}(\bar\mu_0,\bar\mu_1)$ such that
\begin{equation}\label{eq:bdd-energy}
 \sup\limits_n\int_0^T\cA(\mu_t^n,\bnu_t^n) \dd t~<~\infty\ .
\end{equation}
Then there exists a couple
$(\mu,\bnu)\in\CE_{T}(\bar\mu_0,\bar\mu_1)$ such that up to
extraction of a subsequence
\begin{align*}
 & \mu^n_t\rightharpoonup\mu_t\quad \mbox{weakly in }\cP(\R^d)\ \mbox{for all } t\in[0,T]\ ,\\
 & \bnu^n\rightharpoonup^*\bnu\quad \mbox{weakly* in }\cM(G\times(0,T))\ .
\end{align*}
Moreover along this subsequence we have :
\begin{equation*}
 \int_0^T\cA(\mu_t,\bnu_t) \dd t~\leq~\liminf\limits_n \int_0^T\cA(\mu_t^n,\bnu_t^n) \dd t\ .
\end{equation*}
\end{proposition}
\begin{proof}
  For each $n$ define the measure $\bnu^n:=\int_0^T\bnu^n_t\dd
  t\in\cM_{loc}(G\times(0,T))$.  From Lemma \ref{lem:integrability}
  and \eqref{eq:bdd-energy} we infer immediately that
  \begin{align}\label{eq:unif-integrability1}
    \sup\limits_n\int_0^T\int\big(1\wedge\abs{x-y}\big)\abs{\bnu^n}(\dd x,\dd y)\dd t~<~\infty\ .
  \end{align}
  Moreover, for every compact set $K\subset G$ we obtain
  \begin{align}\label{eq:unif-integrability2}
    \sup\limits_n\abs{\bnu^n}(K\times[0,T])~\leq~\sup\limits_n\int_0^T\abs{\bnu^n_t}(K)\dd t~<~\infty\ .
  \end{align}
  i.e. $\bnu^n$ has total variation uniformly bounded on every compact
  subset of $G\times[0,T]$. Hence we can extract a subsequence (still
  indexed by $n$) such that $\bnu^n\rightharpoonup^*\bnu$ in
  $\cM_{loc}(G\times[0,T])$. By the disintegration theorem we have the
  representation $\bnu=\int_0^T\bnu_t\dd t$ for a Borel family
  $(\bnu_t)$ still satisfying \eqref{eq:cefinitemass}. Let us set
  $D=\{(x,x)\ :\ x\in\R^d\}$ and define the finite measures
  $\tilde\bnu^n\in\cM(\R^{2d}\times[0,T])$ given by $\tilde\bnu^n(\dd
  x,\dd y)=(1\wedge\abs{x-y})\bnu^n(\dd x,\dd y)\dd t$ on
  $G\times[0,T]$ and
  $\tilde\bnu^n(D\times[0,T])=0$. \eqref{eq:unif-integrability1}
  implies that (up to extraction of another subsequence)
  $\tilde\bnu^n\rightharpoonup^*\tilde\bnu$ in
  $\cM(\R^{2d}\times[0,T])$ where $\tilde\bnu$ is defined similar to
  $\tilde\bnu^n$.
  
  Let $0\leq t_0\leq t_1\leq T$ and $\xi\in C^\infty_c(\R^d)$. We
  claim that
  \begin{align}\label{eq:converge-bnu}
    \int_{t_0}^{t_1}\int\dgrad\xi \dd\bnu^n_t
    \dd t~\overset{n\to\infty}{\longrightarrow}~\int_{t_0}^{t_1}\int\dgrad\xi
    \dd\bnu_t \dd t\ .
  \end{align}
  Let us define $\beta:\R^{2d}\times[0,T]\to\R$ by setting
  \begin{align*}
    \beta(x,y,t)=
    \begin{cases}
      \one_{(t_0,t_1)}(t)\dgrad\xi(x,y)(1\wedge\abs{x-y})^{-1}\;,& x\neq y\;,\\
      0\;, & x=y\;.
    \end{cases}
  \end{align*}
  Now \eqref{eq:converge-bnu} is equivalent to $\int\beta\dd
  \tilde\bnu^n\to\int\beta\dd\tilde\bnu$. Note that $\beta$ is bounded
  with compact support and that the discontinuity set of $\beta$ is
  concentrated on $\R^{2d}\times\{t_0,t_1\}\cup D\times[0,T]$ which is
  negligible for $\tilde\bnu$. Hence the claim follows from general
  convergence results (see e.g. \cite[Prop. 5.1.10]{AGS05}).
  
  Combining now the convergence \eqref{eq:converge-bnu} with
  \eqref{eq:cedistributionrefined} for $\varphi(t,x)=\xi(x)$ and
  $t_0=0,t_1=t$ we infer that $\mu^n_t$ converges weakly to some
  $\mu_t\in\cP(\R^d)$ for every $t\in[0,T]$. It is easily checked that
  the couple $(\mu,\bnu)$ belongs to
  $\CE_{T}(\bar\mu_0,\bar\mu_1)$. As in Lemma \ref{lem:lscaction} the
  lower semicontinuity now follows from Proposition \ref{prop:lsc} by
  considering $\int_0^T\cA(\mu_t,\bnu_t)\dd t$ as an integral
  functional on the space $\cM_{loc}(G\times[0,T])$.
\end{proof}

\section{A non-local transport distance}\label{sec:metric}

We are now ready to give the definition of the distance $\cW$. We will
then establish various properties, in particular existence of
geodesics. Moreover, we will characterize absolutely continuous curves
in the metric space $(\cP,\cW)$.

\begin{definition}\label{def:metric}
 For $\bar\mu_0,\bar\mu_1\in\cP(\R^d)$ we define
\begin{equation}\label{eq:defmetric}
\cW(\bar\mu_0,\bar\mu_1)^2~:=~\inf\left\{\int_0^1\cA(\mu_t,\bnu_t)\dd t\ :\quad
(\mu,\bnu)\in\CE_{1}(\bar\mu_0,\bar\mu_1)\right\}\ .
\end{equation}
\end{definition}

Let us first give an equivalent characterization of the infimum in
\eqref{eq:defmetric}.

\begin{lemma}\label{lem:equivcharacter}
For any $T>0$ and $\bar\mu_0,\bar\mu_1\in\cP(\R^d)$ we have :
\begin{equation}\label{eq:equivcharacter}
\cW(\bar\mu_0,\bar\mu_1)~=~\inf\left\{\int_0^T
\sqrt{\cA(\mu_t,\bnu_t)}\dd t\ :\quad
(\mu,\bnu)\in\CE_{T}(\bar\mu_0,\bar\mu_1)\right\}\ .
\end{equation}
\end{lemma}

\begin{proof}
  This follows from a standard reparametrization argument. See
  \cite[Lem. 1.1.4]{AGS05} or \cite[Thm. 5.4]{DNS09} for details in
  similar situations.
\end{proof}

The next result shows that the infimum in the definition above is
in fact a minimum.

\begin{proposition}\label{prop:minimizers}
  Let $\bar\mu_0,\bar\mu_1\in\cP(\R^d)$ be such that
  $W:=\cW(\bar\mu_0,\bar\mu_1)$ is finite. Then the infimum in
  \eqref{eq:defmetric} is attained by a curve
  $(\mu,\bnu)\in\CE_{1}(\bar\mu_0,\bar\mu_1)$ satisfying
  $\cA(\mu_t,\bnu_t)=W^2$ for a.e. $t\in[0,1]$.
\end{proposition}

\begin{proof}
  Existence of a minimizing curve
  $(\mu,\bnu)\in\CE_{1}(\bar\mu_0,\bar\mu_1)$ follows immediately by
  the direct method taking into account Proposition
  \ref{prop:cecompactness}. Invoking Lemma \ref{lem:equivcharacter}
  and Jensen's inequality we see that this curve satisfies
  \begin{align*}
    \int_0^1\sqrt{\cA(\mu_t,\bnu_t)}\dd
    t~\geq~W~=~\left(\int_0^1\cA(\mu_t,\bnu_t)\dd
      t\right)^{\frac{1}{2}}~\geq~\int_0^1\sqrt{\cA(\mu_t,\bnu_t)}\dd t\ .
  \end{align*}
 Hence we must have $\cA(\mu_t,\bnu_t)=W^2$ for a.e. $t\in[0,T]$.
\end{proof}
 
We now prove the first main result Theorem \ref{thm:1-metric}
announced in the introduction which we recall here for convenience.

\begin{theorem}\label{thm:distance}
  $\cW$ defines a (pseudo-) metric on $\cP(\R^d)$. The topology it
  induces is stronger than the weak topology and bounded sets
  w.r.t. $\cW$ are weakly compact. Moreover, the map
  $(\mu_0,\mu_1)\mapsto \cW(\mu_0,\mu_1)$ is lower semicontinuous
  w.r.t. weak convergence. For each $\tau\in\cP(\R^d)$ the set
  $\cP_\tau:=\{\mu\in\cP(\R^d)\ :\ \cW(\mu,\tau)<\infty\}$ equipped
  with the distance $\cW$ is a complete geodesic space.
\end{theorem}

\begin{proof}
  Symmetry of $\cW$ is obvious from the fact that
  $\a(w,\cdot,\cdot)=\a(-w,\cdot,\cdot)$. Equation
  \eqref{eq:cedistributionrefined} from Lemma \ref{lem:cont-represent}
  shows that two curves in $\CE_{1}$ can be concatenated to obtain a
  curve in $\CE_{2}$. Hence the triangle inequality follows easily
  using Lemma \ref{lem:equivcharacter}. To see that
  $\cW(\bar\mu_0,\bar\mu_1)>0$ whenever $\bar\mu_0\neq\bar\mu_1$
  assume that $\cW(\bar\mu_0,\bar\mu_1)=0$ and choose a minimizing
  curve $(\mu,\bnu)\in\CE_{1}(\bar\mu_0,\bar\mu_1)$. Then we must have
  $\cA(\mu_t,\bnu_t)=0$ and hence $\bnu_t=0$ for
  a.e. $t\in(0,1)$. From the continuity equation in the form
  \eqref{eq:cedistributionrefined} we infer $\bar\mu_0=\bar\mu_1$.
  
  Let us now show that the topology induced by $\cW$ is stronger than
  the weak one. Let $\mu_n,\mu\in\cP(\R^d)$ with $\cW(\mu_n,\mu)\to0$
  and choose minimizing curves
  $(\mu^n,\bnu^n)\in\CE_1(\mu_n,\mu)$. Fix a function
  $\varphi:\R^d\to\R$ bounded in $C^1$. Using the continuity equation
  in the form \eqref{eq:cedistributionrefined} and Lemma
  \ref{lem:integrability} we estimate:
  \begin{align*}
      \abs{\int\varphi \dd\mu_n-\int\varphi
        \dd\mu}~&=~\frac12\abs{\int_0^1\int\dgrad\varphi \dd\bnu^n_t\dd t}\\
      &\leq~\norm{\phi}_{C^1} \int_0^1\int\big(1\wedge\abs{x-y}\big)\abs{\bnu^n_t}(\dd x,\dd y)\dd t\\
      &\leq~\norm{\phi}_{C^1}C\int_0^1\sqrt{\cA(\mu^n_t,\bnu^n_t)}\dd t~=~\norm{\phi}_{C^1}C\cdot
      \cW(\mu_n,\mu)\ .
  \end{align*}
  This implies $\mu_n\rightharpoonup\mu$ weakly.

  The compactness assertion and lower semicontinuity of
  $\cW$ follow immediately from Proposition
  \ref{prop:cecompactness}. Let us now fix $\tau\in\cP(\R^d)$ and let
  $\bar\mu_0,\bar\mu_1\in \cP_\tau$. By the triangle inequality we
  have $\cW(\bar\mu_0,\bar\mu_1)<\infty$ and hence Proposition
  \ref{prop:minimizers} yields existence of minimizing curve
  $(\mu,\bnu)\in\CE_{1}(\bar\mu_0,\bar\mu_1)$. The curve
  $t\mapsto\mu_t$ is then a constant speed geodesic in $\cP_\tau$
  since it satisfies
  \begin{align*}
    \cW(\mu_s,\mu_t)~=~\int\limits_s^
    t\sqrt{\cA(\mu_r,\bnu_r)}\dd r~=~(t-s)\cW(\mu_0,\mu_1)\quad\forall
    0\leq s\leq t\leq 1\ .
  \end{align*}
  To show completeness let $(\mu^n)_n$ be a Cauchy sequence in
  $\cP_\tau$. In particular the sequence is bounded w.r.t. $\cW$ and
  we can find a subsequence (still indexed by $n$) and $\mu^\infty\in$
  such that $\mu^n\rightharpoonup^*\mu^\infty$. Invoking lower
  semicontinuity of $\cW$ and the Cauchy condition we infer
  $\cW(\mu^n,\mu^\infty)\to 0$ as $n\to\infty$ and $\mu^\infty\in
  \cP_\tau$.
\end{proof}

It is yet unclear when precisely the distance $\cW$ is
finite. However, we will see in the next section that the distance is
finite e.g. along trajectories of the semigroup associated to a
translation invariant jump kernel.

The following result shows that under certain assumptions the
distance $\cW$ can be bounded from below by the $L^1$-Wasserstein
distance. Recall that this distance is defined for
$\mu_0,\mu_1\in\cP(\R^d)$ by
\begin{align*}
  W_1(\mu_0,\mu_1)~:=~\inf\limits_\pi \int\limits_{\R^d\times\R^d}\abs{x-y}\pi(\dd x,\dd y)\ ,
\end{align*}
where the infimum is taken over all probability measures
$\pi\in\cP(\R^d\times\R^d)$ whose first and second marginal are
$\mu_0$ and $\mu_1$ respectively (see e.g. \cite[Chap. 6]{Vil09}).
\begin{proposition}\label{prop:lowerboundW}
  Assume that the jump kernel $J$ satisfies
  \begin{align}\label{eq:uniform-momentsJ}
    M^2~:=~\sup\limits_x\int\abs{x-y}^2J(x,\dd y)~<~\infty\ .
  \end{align}
  Then for any $\mu_0,\mu_1\in\cP(\R^d)$ we have the bound
  \begin{align*}
    W_1(\mu_0,\mu_1)~\leq~\frac{M}{\sqrt{2}}\cW(\mu_0,\mu_1)\ .
  \end{align*}
\end{proposition}

\begin{proof}
  We can assume that $\cW(\mu_0,\mu_1)<\infty$. Take a minimizing
  curve $(\mu,\bnu)\in\CE_1(\mu_0,\mu_1)$ and let $\phi:\R^d\to\R$ be
  a $1$-Lipschitz function. Using the continuity equation in the for
  \eqref{eq:cedistributionrefined} and arguing similar as in Lemma
  \ref{lem:integrability} we estimate
  \begin{align*}
    &\abs{\int\varphi \dd\mu_n-\int\varphi \dd\mu}\\
    ~&=~\frac12\abs{\int_0^1\int\dgrad\varphi \dd\bnu^n_t\dd t}\\
    &\leq~\frac12\int_0^1\int\abs{x-y}\abs{\bnu^n_t}(\dd x,\dd y)\dd t\\
    &\leq~\frac{1}{\sqrt{2}}\left(\int_0^1\cA(\mu^n_t,\bnu^n_t)\dd t\right)^\frac12\left(\int_0^1\int\abs{x-y}^2J(x,\dd y)\mu_t(\dd x)\dd t\right)^\frac12\\
    &\leq~\frac{M}{\sqrt{2}}\cW(\mu_n,\mu)\ .
  \end{align*}
  Taking the supremum over all $1$-Lipschitz functions $\phi$ yields
  the claim by Kantorovich-Rubinstein duality (see
  e.g. \cite[5.16]{Vil09}).
\end{proof}

We now give a characterization of absolutely continuous curves with
respect to $\cW$ and relate their length to their minimal
action. Recall that a curve $(\mu_t)_{t\in[0,T]}$ in $\cP(\R^d)$ is
called absolutely continuous w.r.t. $\cW$ if there exists $m\in
L^1(0,T)$ such that
\begin{align}\label{eq:abs-continuous}
  \cW(\mu_s,\mu_t)~\leq~\int_s^tm(r)\dd r \quad\forall~0\leq s\leq t\leq
  T\ .
\end{align}
For an absolutely continuous curve the metric derivative defined by
\begin{equation*}
  \abs{\mu_t'}~:=~\lim\limits_{h\to0}\frac{\cW(\mu_{t+h},\mu_t)}{\abs{h}}
\end{equation*}
exists for a.e. $t\in[0,T]$ and is the minimal $m$ in
\eqref{eq:abs-continuous}.

\begin{proposition}[Metric velocity]\label{prop:metricderivative}
  A curve $(\mu_t)_{t\in[0,T]}$ is absolutely continuous with respect
  to $\cW$ if and only if there exists a Borel family
  $(\bnu_t)_{t\in[0,T]}$ such that $(\mu,\bnu)\in\CE_{T}$ and
  \begin{align*}
    \int_0^T\sqrt{\cA(\mu_t,\bnu_t)}dt~<~\infty\ .
  \end{align*}
  In this case we have $\abs{\mu_t'}^2\leq\cA(\mu_t,\bnu_t)$ for
  a.e. $t\in[0,T]$. Moreover, there exists a unique Borel family
  $\tilde{\bnu}_t$ with $(\mu,\tilde\bnu)\in\CE_{T}$ such that
  \begin{align}\label{eq:optimalvelocity}
    \abs{\mu_t'}^2=\cA(\mu_t,\tilde{\bnu}_t)\qquad \text{for a.e. }t\in[0,T]\ .
  \end{align}
\end{proposition}

\begin{proof}
  The proof follows from the very same arguments as in
  \cite[Thm. 5.17]{DNS09}.
\end{proof}

We can describe the optimal velocity measures $\tilde\bnu_t$ appearing
in the preceding proposition in more detail.  We define
\begin{align}\label{eq:bnu-tangent}
T_\mu\cP(\R^d)~:=~\Big\{&\bnu\in\cM_{loc}(G)\ :\ \cA(\mu,\bnu)<\infty\;,\\\nonumber
&\cA(\mu,\bnu)~\leq~\cA(\mu,\bnu+\bbeta)\ \forall \bbeta\;:\;\dgrad\cdot\bbeta=0\Big\}\ .
\end{align}
Here $\dgrad\cdot\bbeta=0$ is understood in a weak sense, i.e.
\begin{align*}
\frac12\int\dgrad\xi(x,y)\bbeta(\dd x,\dd y)~=~0\qquad\forall\xi\in C^\infty_c(\R^d)\ .
\end{align*}

\begin{corollary}\label{cor:tangentspace}
  Let $(\mu,\bnu)\in\CE_{T}$ such that the curve $t\mapsto\mu_t$ is
  absolutely continuous w.r.t. $\cW$. Then $\bnu$ satisfies
  \eqref{eq:optimalvelocity} if and only if $\bnu_t\in
  T_{\mu_t}\cP(\R^d)$ for a.e. $t\in[0,T]$.
\end{corollary}

In the light of the formal Riemannian interpretation of the distance
$\cW$ we view $T_\mu\cP(\R^d)$ as the tangent space to $\cP(\R^d)$ at
the measure $\mu$. If $\mu$ is absolutely continuous with respect to
$m$ we can give an explicit description of $T_\mu\cP(\R^d)$ as a
subspace of an $L^2$ space. For this recall that we denote by
$Jm\in\cM_{loc}(G)$ the measure given by $Jm(\dd x,\dd y)=J(x,\dd
y)m(\dd x)$.

\begin{proposition}\label{prop:tangentspace-density}
  Let $\mu=\rho m\in\cP(\R^d)$. Then we have $\bnu\in T_\mu\cP(\R^d)$
  if and only if $\bnu=w\hat\rho Jm$ is absolutely continuous
  w.r.t. the measure $\hat\rho Jm$ and
  \begin{align*}
    w~\in~ \overline{\{\dgrad\phi\ \vert\ \phi\in C^\infty_c(\R^d)\}}^{L^2(\hat\rho Jm)}~=:~T_\rho\ .
  \end{align*}
\end{proposition}

\begin{proof}
  If $\cA(\mu,\bnu)$ is finite we infer from Lemma \ref{lem:densities}
  that $\bnu=w\hat\rho Jm$ for some density $w:G\to\R$ and that
  $\cA(\mu,\bnu)=\norm{w}^2_{L^2(\hat\rho Jm)}$. Now the optimality
  condition in \eqref{eq:bnu-tangent} is equivalent to
  \begin{align*}
    \norm{w}_{L^2(\hat\rho Jm)}~\leq~\norm{w+v}_{L^2(\hat\rho Jm)}\qquad \forall v\in N_\rho\ ,
  \end{align*}
  where $N_\rho:=\{v\in L^2(\hat\rho Jm)\;:\;\int\dgrad\xi
  v\hat\rho\;\dd Jm = 0\ \forall \xi\in C^\infty_c(\R^d)\}$. This
  implies the assertion of the proposition after noting that $N_\rho$
  is the orthogonal complement in $L^2$ of $T_\rho$.
\end{proof}

The convexity and monotonicity properties of the action functional
established in Section \ref{sec:action} extend naturally to the
distance function.

\begin{proposition}[Convexity of the distance]\label{prop:convexitydistance}
  Let $\mu_0^j,\mu_1^j\in\cP(\R^d)$ for $j=0,1$. For $\tau\in[0,1]$
  and $k=0,1$ set $\mu_k^\tau=\tau\mu_k^1+(1-\tau)\mu_k^0$. Then we
  have :
 $$\cW(\mu^\tau_0,\mu^\tau_1)^2~\leq~\tau\cW(\mu^1_0,\mu^1_1)^2+(1-\tau)\cW(\mu^0_0,\mu^0_1)^2\ . $$
\end{proposition}

\begin{proof}
  We can assume that $\cW(\mu^j_0,\mu^j_1)$ is finite and choose
  minimizing curves
  $(\mu^j,\bnu^j)\in\CE_{1}(\mu^j_0,\mu^j_1)$. Then for
  $t\in[0,1]$ set $\mu_t^\tau=\tau\mu_t^1+(1-\tau)\mu_t^0$ and
  $\bnu_t^\tau=\tau\bnu_t^1+(1-\tau)\bnu_t^0$. Observe that
  $(\mu^\tau,\bnu^\tau)_t\in\CE_{1}(\mu^\tau_0,\mu^\tau_1)$. From
  the definition of $\cW$ and the convexity of $\cA$ as stated in
  Lemma \ref{lem:convexityPhi} we infer
  \begin{equation*}
    \begin{split}
      \cW(\mu^\tau_0,\mu^\tau_1)^2~&\leq~\int_0^1\cA(\mu_t^\tau,\bnu^\tau_t)\dd t~\leq~\int_0^1\tau\cA(\mu_t^1,\bnu^1_t)+(1-\tau)\cA(\mu_t^0,\bnu^0_t)\dd t\\
      &=~\tau\cW(\mu^1_0,\mu^1_1)^2+(1-\tau)\cW(\mu^0_0,\mu^0_1)^2\ .
    \end{split}
  \end{equation*}
 \end{proof}

 \begin{proposition}[Monotonicity under
   convolution]\label{prop:monotonconvolutionW}
   Let $\mu_0,\mu_1\in\cP(\R^d)$. Assume that $J$ satisfies
   \eqref{eq:translationinv-J} and let $m$ be Lebesgue measure. Let
   $k$ be a convolution kernel. Then we have
  \begin{align*}
    \cW(\mu_0*k,\mu_1*k)~\leq~\cW(\mu_0,\mu_1)\ .
  \end{align*}
  If we set $k_\eps(x)=\eps^{-d}k(x/\eps)$, then as
  $\eps\searrow 0$ we have
  \begin{align*}
    \cW(\mu_0*k_\eps,\mu_1*k_\eps) ~\longrightarrow~\cW(\mu_0,\mu_1)\ .
  \end{align*}
\end{proposition}

\begin{proof}
  Assume that $\cW(\mu_0,\mu_1)$ is finite, as otherwise there is
  nothing to proof. Let $(\mu,\bnu)\in\CE_{1}(\mu_0,\mu_1)$ be a
  minimizing curve according to Proposition
  \ref{prop:minimizers}. Define
  $\tilde\mu_t=\mu_t*k,\tilde\bnu_t=\bnu_t*k$. We claim that
  $(\tilde\mu,\tilde\bnu)\in\CE_{1}(\mu_0*k,\mu_1*k)$.
  Indeed, let us show that the continuity equation (v) in
  \eqref{eq:ce-conditions} holds for $(\tilde\mu,\tilde\bnu)$. The
  other properties are equally easy to verify. So let $\phi\in
  C^\infty_c((0,1)\times\R^d)$ and set
  $\tilde\phi(t,x)=\int\phi(t,x+z)k(z)\dd z$. Using the continuity
  equation for $(\mu,\bnu)$ and \eqref{eq:convolutiondiagonal} we
  obtain
  \begin{align*}
    \int\partial_t\phi\dd \tilde\mu_t\dd t~&=~\int\partial_t\phi(t,x+z)k(z)\dd z\mu_t(\dd x)\dd t\\
    &=~\int\partial_t\tilde\phi\dd\mu_t\dd t~=~-\frac12\int\dgrad\tilde\phi\dd\bnu_t\dd t\\
    &=~-\frac12\int\dgrad\phi(t,x+z,y+z)k(z)\bnu_t(\dd x,\dd y)\dd z\dd t\\
    &=~-\frac12\int \dgrad\phi\dd\tilde\bnu_t\dd t\ .
  \end{align*}
  Now the first assertion follows immediately from Proposition
  \ref{prop:monotonconvolution}. This in turn together with weak lower
  semicontinuity of $\cW$ (see Theorem \ref{thm:distance}) yields the
  second assertion.
\end{proof}

\section{Geodesic convexity and gradient flow of the entropy}\label{sec:gradflow}

In this section we focus on a translation invariant jump kernel $J$
and will identify the evolution equation \eqref{eq:general-evo} as the
gradient flow of the relative entropy in the framework of gradient
flows in metric spaces developed in \cite{AGS05}. So let us assume
from now on that $J$ satisfies
\begin{align*}%\label{eq:translationinv-J}
  J(x-z,A)~=~J(x,A+z)\qquad \forall x,z\in\R^d, A\in\cB(\R^d)
\end{align*}
and that $m$ is Lebesgue measure on $\R^d$. Moreover we assume that
$\theta$ is the logarithmic mean defined by \eqref{eq:log-mean}. Under
this assumptions we can write
\begin{align*}
  J(x,A)=\nu(A-x)\qquad \forall x\in\R^d\;,\ A\in\cB(\R^d)\ ,
\end{align*}
where $\nu$ is a L\'evy measure, i.e. a Borel measure on
$\R^d\setminus\{0\}$ satisfying
\begin{align*}
  \int(1\wedge\abs{y}^2)\;\nu(\dd y)~<~\infty\ .
\end{align*}
Now the evolution equation takes the form 
\begin{align*}
  \partial_t\rho~=~\cL\rho\ ,
\end{align*}
where the operator $\cL$ is given by
\begin{align*}
  \cL\rho(x)~:=~\int\big(\rho(x+y)-\rho(x)-y\cdot\nabla\rho(x)\one_{\{\abs{y}\leq1\}}\big)\nu(\dd y)\ .
\end{align*}
Note that $\cL$ is also the generator of the L\'evy process $X$ with
vanishing drift and diffusion and with L\'evy measure $\nu$ (see
e.g. \cite{App04} for background on L\'evy processes). It is a pseudo
differential operator whose symbol is given by the L\'evy-Khinchine
formula
\begin{align*}
  \eta(\xi)~=~\int e^{i\langle
    y,\xi\rangle}-1-i\langle y,\xi\rangle\one_{\{\abs{y}\leq
    1\}}\nu(\dd y)\ .
\end{align*}
This means that $\F(\cL\rho)=\eta\F(\rho)$, where $\F$ denotes the
Fourier transform. Recall that the law of $X_t$ can be given
explicitly in terms of its Fourier transformation. Namely, we have
\begin{align*}
  \E\big[\exp(i\langle\xi, X_t\rangle)\big]~=~\exp(t\eta(\xi))\ .
\end{align*}
Throughout this section we will make the following assumption on $\nu$
in terms of the law of the associated L\'evy process.
\begin{assumption}\label{ass:regularity}
  Assume that the law of the process $X_t$ has a density $\psi_t$ such
  that $\psi_t>0$ for all $t>0$. Moreover, assume that
  $\psi:(0,\infty)\times\R^d\to\R_+$ is such that $\psi_t,\cL\psi_t$
  are rapidly decreasing functions locally uniformly in $t$.
\end{assumption}
\begin{remark}\label{rem:regularity}
  This is a technical assumption made to simplified the
  presentation. It is used to ensure convergence of integrals in the
  proof of Theorem \ref{thm:EVIforJP} and could be weakened
  substantially. Still, Assumption \ref{ass:regularity} is fulfilled
  for example, when $\nu(\dd y)=c_\alpha\abs{y}^{-\alpha-d}$ for
  $\alpha\in(0,2)$. For a suitable constant $c_\a$ the L\'evy process
  $X$ is then the symmetric, isotropic $\alpha$-stable process and the
  symbol is given by $\eta(\xi)=\abs{\xi}^\alpha$.
\end{remark}
Recall that a smooth function $f:\R^d\to\R$ is called rapidly
decreasing if $\abs{x^\beta D^\alpha f(x)}\to 0$ as $\abs{x}\to\infty$
for any multi-indices $\a,\b$. We obtain a semigroup $(P_t)_{t\geq 0}$
on $\cP(\R^d)$ endowed with the distance $\cW$ by setting
\begin{align*}
  P_t[\mu]~:=~\mu*\psi_t\ .
\end{align*}
For $\bnu\in\cM(G)$ we set 
\begin{align*}
    P_t[\bnu]~:=~\bnu*\psi_t\ ,
\end{align*}
with the convolution being understood in the sense of
\eqref{eq:convolutiondiagonal}. Proposition
\ref{prop:monotonconvolutionW} shows that $P$ is a $C^0$-semigroup in
the sense that $P_t[\mu]\rightharpoonup\mu$ weakly as $t\to
0$. Moreover, $P_t[\mu]=\rho_tm$ is absolutely continuous
w.r.t. Lebesgue measure for any $\mu\in\cP(\R^d)$ and the density
$\rho_t$ satisfies $\partial_t\rho_t=\cL\rho_t$.

The notion of gradient flow can be defined in abstract metric spaces
and has been studied extensively in this setting (see
\cite{AGS05}). Of particular interest are gradient flows of
functionals that are geodesically (semi-) convex. In this situation
the gradient flow is characterized by the so called ``Evolution
Variational Inequality''(EVI). We adopt the following definition.

\begin{definition}\label{def:gradflow}
  Let $(X,d)$ be a metric space and $F:X\to (-\infty,\infty]$ a lower
  semicontinuous function. Further let $(S_t)_{t\geq 0}$ be a
  $C^0$-semigroup on $X$ and $\lambda\in\R$. $S$ is called the
  ($\lambda$-)gradient flow of $F$ if $S_t(X)\subset D(F)$ for all
  $t>0$, the map $t\mapsto F(S_t(u))$ is non-increasing for all $u\in
  X$ and if for all $u\in X,v\in D(F), t\geq 0$:
  \begin{align}\label{eq:EVI}
    \frac12\frac{\dd^+}{\dd t}d^2(S_t(u),v) + \frac{\lambda}{2}d^2(S_t(u),v) + F(S_t(u))~\leq~F(v)\ .
  \end{align}
\end{definition}
Here $D(F):=\{x\in X\ \vert\ F(x)<\infty\}$ denotes the proper domain
of the function $F$.

 We will apply this definition in the case where
$X=\cP(\R^d)$ and $F$ is the relative entropy $\cH$ defined for
$\mu\in\cP(\R^d)$ by
\begin{align*}
  \cH(\mu)~:=~
  \begin{cases}
    \int\rho\log\rho\; \dd m\ , & \text{if }\mu=\rho \dd m\ \text{and } \int(\rho\log\rho)_+\dd m<\infty\\
    +\infty\ , & \text{else.} 
  \end{cases}
\end{align*}
Let us start by stating a result giving the entropy production along
the semigroup $P$. As before, we will denote by $Jm\in\cM_{loc}(G)$
the measure given by $Jm(\dd x,\dd y)=J(x,\dd y)m(\dd x)$. For a
probability measure $\mu\in\cP(\R^d)$ we define a non-local analogue
of the Fisher information by
\begin{align}\label{eq:nonlocalFisher-information}
  \cI(\mu)~:=~
  \begin{cases}
    \frac12\int\dgrad\rho\dgrad\log\rho\; \dd(Jm) , & \text{if }\mu=\rho m \text{ and } \rho>0\ ,\\
  +\infty\ , & \text{else} \ .
  \end{cases}
\end{align}

\begin{proposition}\label{prop:entropy-dissipation}
  Let $\mu\in\cP(\R^d)$ and set $\mu_t=\rho_t m:=P_t[\mu]$. For every
  $t>0$ we have $\cH(\mu_t)\in(-\infty,\infty)$ and
  $\cI(\mu_t)<\infty$. Moreover, we have the energy identity
  \begin{align}\label{eq:entropy-dissipation}
    \cH(\mu_t) - \cH(\mu_s)~=~-\int_s^t\cI(\mu_r)\;\dd r\quad\forall  t\geq s>0\ .
  \end{align}
In particular the map $t\mapsto\cH(\mu_t)$ is non-increasing.
\end{proposition}

\begin{proof}
  Finiteness of $\cH(\mu_t)$ follows readily from the fact that
  $\psi_t$ is rapidly decreasing. We prove
  \eqref{eq:entropy-dissipation} by approximating $\cH$ with
  functionals $\cH_n$. Let us set
  \begin{align}\label{eq:approx-log}
    f_n(u):=\int_0^u\max(1+\log(r),-n)\;\dd r\ .
  \end{align}
  Then we have $f_n(u)\searrow u\log(u)$ and $f_n'(u)\searrow
  1+\log(u)$ as $n\to\infty$. For $\mu=\rho m\in\cP(\R^d)$ we set
  $\cH_n(\mu):=\int f_n(\rho)\dd m$. Now we calculate
  \begin{align*}
    \cH_n(\mu_t) - \cH_n(\mu_s)~&=~\int f_n(\rho_t)-f_n(\rho_s)\;\dd m\\
    &=~\int\int_s^t f_n'(\rho_r)\partial_r\rho_r\;\dd r\dd m~=~\int\int_s^t f_n'(\rho_r)\cL\rho_r\;\dd r\dd m\\
    &=~\frac12\int_s^t\int\dgrad f_n'(\rho_r)\dgrad\rho_r\;\dd(Jm)\dd r \ .
  \end{align*}
  The interchange of integrals and integration by parts are easily
  justified by the fact that $f_n'(\rho_r)$ is bounded and $\cL\rho_r$
  is rapidly decreasing locally uniformly in $r$. Letting finally
  $n\to\infty$ we obtain \eqref{eq:entropy-dissipation} by monotone
  convergence of both the left and right hand sides.
\end{proof}

We will now show that the semigroup $(P_t)$ is the gradient flow of
the relative entropy with respect to the distance $\cW$. Our strategy
of proof is inspired by an argument developed in \cite{DS08} and used
in a similar form in \cite[Thm. 5.29]{DNS09}. Recall that $\cW$ is a
pseudo distance, thus it is necessary to consider the sets
$\cP_\tau:=\{\mu\in\cP(\R^d) : \cW(\mu,\tau)<\infty\}$ for a
given $\tau\in\cP(\R^d)$. The following two results are a restatement
of Theorem \ref{thm:2-gradflow}.

\begin{theorem}\label{thm:EVIforJP}
  Let $\mu\in\cP(\R^d)$ and set $\mu_t:=P_t[\mu]$. Then $\mu_t\in
  D(\cH)\cap\cP_\mu$ for all $t>0$ and the map $t\mapsto\cH(\mu_t)$ is
  non-increasing. Moreover, for any $\sigma\in\cP_\mu$ the
  Evolution Variational Inequality holds:
  \begin{align}\label{eq:EVI0}
    \frac12\frac{\dd^+}{\dd t}\cW^2(\mu_t,\sigma) + \cH(\mu_t)~\leq~\cH(\sigma)\quad\forall t>0\ .
  \end{align}
\end{theorem}

\begin{proof}
  The first statement is a direct consequence of Proposition
  \ref{prop:entropy-dissipation}. For the second statement it is
  sufficient to assume $\mu\in D(\cH)$ and prove the inequality at
  $t=0$. So let $\sigma\in D(\cH)$ and let
  $(\mu_s,\bnu_s)_{s\in[0,1]}$ be a minimizing curve $\mu_0:=\sigma$
  to $\mu_1:=\mu$. We set
  \begin{align*}
  \mu^\eps_{s,t}~&=~\rho^\eps_{s,t}m~:=~P_{st+\eps}[\mu_s] \quad\text{and}\\
  \tilde\bnu^\eps_{s,t}~&=~\tilde v^\eps_{s,t}Jm~:=~P_{st+\eps}[\bnu_s]\ . 
  \end{align*}
  The couple $(\mu^\eps_{s,t},\tilde\bnu^\eps_{s,t})$ does not satisfy
  the continuity equation. Hence we make the correction
  \begin{align*}
    \bnu^\eps_{s,t}=v^\eps_{s,t}Jm:=(\tilde v^\eps_{s,t}-t\dgrad\rho^\eps_{s,t})Jm\ .
  \end{align*}
  We will need the following result whose proof we
  postpone for the moment.
  \begin{claim}\label{cl:EVIforJP}
    We have
    $(\mu^\eps_{\cdot,t},\bnu^\eps_{\cdot,t})\in\CE_1(\sigma_\eps,\mu_{\eps+t})$
    and moreover,
    \begin{align}\label{eq:entropy-dissipation-special}
      \cH(\mu_{\eps+t}) -
      \cH(\mu_\eps)~=~-\frac12\int_0^1\int\dgrad\log\rho^\eps_{s,t}v^\eps_{s,t}\;\dd(Jm)\;\dd s\ .
    \end{align}
  \end{claim}
  From the definition of the distance $\cW$ we now obtain the estimate
  \begin{align}\label{eq:EVIforJP-1}
    \cW(\mu_{t+\eps},\sigma_\eps)^2~\leq~\int_0^1\cA(\mu^\eps_{s,t},\bnu^\eps_{s,t})\;\dd s\ .
  \end{align}
  Recall the notation $\hat\rho(x,y)=\theta(\rho(x),\rho(y))$ with
  $\theta$ being the logarithmic mean here. We can further estimate
  \begin{align*}%\label{eq:EVIforJP-2}
    \cA(\mu^\eps_{s,t},\bnu^\eps_{s,t})~&=~\int\frac{\abs{v^{\eps}_{s,t}}^2}{2\hat\rho^\eps_{s,t}}\;\dd(Jm)\\
    &=~\int\big(\abs{\tilde v^{\eps}_{s,t}}^2-2t\dgrad\rho^\eps_{s,t}v^\eps_{s,t}-t^2\abs{\dgrad\rho^\eps_{s,t}}^2\big)  \frac{1}{2\hat\rho^\eps_{s,t}}\;\dd(Jm)\\
    &\leq~\cA(\mu^\eps_{s,t},\tilde\bnu^\eps_{s,t}) - t \int\dgrad\log\rho^\eps_{s,t}v^\eps_{s,t}\;\dd(Jm)\\
    &\leq~\cA(\mu_s,\bnu_s) - t \int\dgrad\log\rho^\eps_{s,t}v^\eps_{s,t}\;\dd(Jm)\ ,
  \end{align*}
  where we have dropped the quadratic term in $t$ and used the
  monotonicity under convolution (Proposition
  \ref{prop:monotonconvolution}) in the last inequality. Integration
  over $s$ from $0$ to $1$ and using
  \eqref{eq:entropy-dissipation-special} gives
 \begin{align*}
   \frac12\cW(\mu_{t+\eps},\sigma_\eps)^2~\leq~\frac12\cW(\mu,\sigma)^2 - t\cdot\big(\cH(\mu_{t+\eps})-\cH(\sigma_\eps)\big)\ .
 \end{align*}
 By lower semicontinuity of $\cW$ (see Theorem \ref{thm:distance}) and
 continuity of $\cH$ along the semigroup we can take the limit
 $\eps\to 0$ and obtain
 \begin{align*}
   \frac12\cW(\mu_{t},\sigma)^2~\leq~\frac12\cW(\mu,\sigma)^2 - t\cdot\big(\cH(\mu_{t})-\cH(\sigma)\big)\ .
 \end{align*}
 Finally, rearranging terms and letting $t\searrow 0$ yields
 \eqref{eq:EVI0}.
 \begin{proof}[Proof of Claim \ref{cl:EVIforJP}]
   For the proof we first need two estimates. First note that 
   \begin{align}\label{eq:finite-I}
     \int_0^1\cI(\mu^\eps_{s,t})\;\dd s < \infty\ .
   \end{align}
   Indeed, by convexity of the map $(u,v)\mapsto (u-v)(\log u-\log v)$
   we have that $\cI(\mu*\psi_t)\leq\cI(\psi_tm)$ for every
   $\mu\in\cP(\R^d)$. Hence we conclude from Proposition
   \ref{prop:entropy-dissipation} that
   \begin{align*}
    \int_0^1\cI(\mu^\eps_{s,t})\;\dd s~\leq~\int_0^1\cI(\psi_{\eps+st}m)\;\dd s~=~\cH(\psi_{\eps}m)-\cH(\psi_{\eps+t}m)~<~\infty\ .
   \end{align*}
   From this we conclude that the curve
   $(\mu^\eps_{\cdot,t},\bnu^\eps_{\cdot.t})$ has finite
   action. Indeed,
   \begin{align*}
     A~&:=~\int_0^1\int\frac{\abs{v^\eps_{s,t}}^2}{2\hat\rho^\eps_{s,t}}\;\dd(Jm)\;\dd s~\\
       &\leq~\int_0^1\int 2\frac{\abs{\tilde v^\eps_{s,t}}^2}{2\hat\rho^\eps_{s,t}}+ 2t^2\frac{\abs{\dgrad\rho^\eps_{s,t}}^2}{2\hat\rho^\eps_{s,t}}\;\dd(Jm)\;\dd s\\
    &\leq~2\int_0^1\cA(\mu_s,\bnu_s)\dd s + 2t^2\int_0^1\cI(\mu^\eps_{s,t})\;\dd s ~<~\infty\ ,
   \end{align*}
   where we use Proposition \ref{prop:monotonconvolution} in the last
   inequality. Using Lemma \ref{lem:integrability} and the previous
   estimate we see that $\bnu^\eps_{\cdot,t}$ satisfies the
   integrability condition (iv) in Definition \ref{def:ce}. The other
   conditions are also easily checked. Hence we see
   $(\mu^\eps_{\cdot,t},\bnu^\eps_{\cdot,t})\in\CE_1(\sigma_\eps,\mu_{\eps+t})$.

   Now let us prove \eqref{eq:entropy-dissipation-special}. By a
   simple convolution argument we can assume that $\rho^\eps_{s,t}$ is
   differentiable in $s$. Let $f_n$ be the function defined by
   \eqref{eq:approx-log} and set $f(u)=u\log(u)$ for $u\geq0$. Now we
   calculate
   \begin{align*}
    \cH_n(\mu_{\eps +t}) - \cH_n(\mu_{\eps})~=~\int\int_0^1 f_n'(\rho^\eps_{s,t})\partial_s\rho^\eps_{s,t}\;\dd s\;\dd m\ .
  \end{align*}
  Note that the map $x\mapsto f_n'(\rho^\eps_{s,t}(x))$ is bounded and
  Lipschitz uniformly in $s\in[0,1]$. Using the integrability
  condition (iv) from Definition \ref{def:ce} we can approximate it by
  functions in $C^\infty_c((0,1)\times\R^d)$ and obtain by the
  continuity equation
  \begin{align}\label{eq:entropy-dissipation-calc}
    \cH_n(\mu_{\eps +t}) - \cH_n(\mu_{\eps})~=~-\frac12 \int_0^1\int\dgrad f_n'(\rho^\eps_{s,t})v^\eps_{s,t}\dd(Jm)\dd s \ .
  \end{align}
  By monotone convergence the left hand side of
  \eqref{eq:entropy-dissipation-calc} converges to the left hand side
  of \eqref{eq:entropy-dissipation-special}. It remains to prove
  convergence of the right hand side. Using H\"older inequality we
  estimate
  \begin{align*}
    &\abs{\int_0^1\int\dgrad (f'(\rho^\eps_{s,t})- f_n'(\rho^\eps_{s,t}))\dd\bnu^\eps_{s,t}\dd s}\\
    &\leq~\int_0^1\int\abs{\dgrad(f'(\rho^\eps_{s,t})-f_n'(\rho^\eps_{s,t}))}\abs{w^\eps_{s,t}}\dd(Jm)\dd s\\
    &\leq~A^{\frac12}\left(\int_0^1\int \abs{\dgrad(f'(\rho^\eps_{s,t})- f_n'(\rho^\eps_{s,t}))}^22\hat\rho^\eps_{s,t}\dd(Jm)\dd s\right)^{\frac12}\ .
  \end{align*}
 The integrand in the last term is bounded as
 \begin{align*}
   \abs{\dgrad(f'(\rho^\eps_{s,t})-f_n'(\rho^\eps_{s,t}))}^2\hat\rho^\eps_{s,t}~\leq~\abs{\dgrad f'(\rho^\eps_{s,t})}^2\hat\rho^\eps_{s,t}~=~\dgrad\log\rho^\eps_{s,t}\dgrad\rho^\eps_{s,t}\ .
 \end{align*}
 With the help of \eqref{eq:finite-I} and dominated convergence we
 conclude convergence of the right hand side of
 \eqref{eq:entropy-dissipation-calc} to the right hand side of
 \eqref{eq:entropy-dissipation-special}.
 \end{proof}
\end{proof}

\begin{corollary}\label{cor:convex-entropy}
  The entropy is convex along $\cW$-geodesics. More precisely, let
  $\mu_0,\mu_1\in\cP(\R^d)$ such that $\cW(\mu_0,\mu_1)<\infty$ and
  let $(\mu_t)_{t\in[0,1]}$ be a geodesic connecting $\mu_0$ and
  $\mu_1$. Then we have
  \begin{align*}
    \cH(\mu_t)~\leq~(1-t)\cH(\mu_0) + t \cH(\mu_1)\ .
  \end{align*}
\end{corollary}

\begin{proof}
  This is a direct consequence of Theorem \ref{thm:EVIforJP} and the fact,
  proved in \cite[Thm. 3.2]{DS08}, that in a general setting the Evolution
  Variational Inequality implies geodesic convexity.
\end{proof}

We finish by giving an equivalent and more intuitive definition of the
distance $\cW$ in the present setting of a translation invariant jump
kernel $J$. We show that it coincides with $\tilde\cW$ defined in
\eqref{eq:metricnaiv}. We introduce the following shorthand
notation. Given functions $\rho:\R^d\to\R_+$ and $\psi:\R^d\to\R$ we
write
\begin{align*}
 \cA'(\rho,\psi) := \frac12\int \big(\psi(y)-\psi(x)\big)^2\hat\rho(x,y)J(x,\dd y) m(\dd x)\ .
\end{align*}
For two probability densities $\bar\rho_0,\bar\rho_1$ w.r.t. $m$ and
$T>0$ let us denote by $\CE'_{T}(\bar\rho_0,\bar\rho_1)$ the
collection of pairs $(\rho,\psi)$ satisfying the following conditions:
\begin{align} \label{eq:conditions} 
 \left\{ \begin{array}{ll}
{(i)} & \rho : [0,T]\times\R^d\to \R_+ \text{ is measurable}\;;\\ 
{(ii)} &  \rho_t \text{ is a probability density for all $t \in [0,T]$}\;;\\
{(iii)} &  \text{The curve } t\mapsto\mu_t:=\rho_tm\ \text{is weakly continuous}\;; \\
{(iv)} & \psi  : [0,T]\times\R^d \to \R  \text{ is measurable}\;;\\ 
{(v)} & \partial_t\rho_t + \dgrad\cdot(\hat\rho_t\dgrad\psi_t)~=~0\;,\ \rho_0=\bar\rho_0\ ,\ \rho_T=\bar\rho_1\ .\\
\end{array} \right.
\end{align}
Here the continuity equation (v) is understood in the sense that for
every test function $\phi\in C^\infty_c((0,T)\times\R^d)$ we have
\begin{align*}
    \int_0^1\int\partial_t\phi\rho_t\dd m\dd t + \frac12\int_0^1\int\dgrad\phi(x,y)\dgrad\psi_t(x,y)\hat\rho_t(x,y)J(x,\dd y)m(\dd y)\dd t~=~0\ .
\end{align*}

\begin{proposition}\label{prop:equiv-definition}
  Assume that $m$ is Lebesgue measure and that $J(x,\dd y)=j(y-x)\dd
  y$ for a function $j:\R^d\setminus\{0\}\to\R^+$ that is strictly
  positive. Moreover, assume that $J$ satisfies
  \ref{ass:regularity}. Let $\bar\mu_i=\bar\rho_im\in \cP(\R^d)$ for
  $i=0,1$ such that $\cI(\bar\mu_i)$ is finite. Then we have
\begin{align*}
  \cW(\bar\mu_0, \bar\mu_1)^2 = 
  \inf \bigg\{ \int_0^1 \cA'(\rho_t, \psi_t) \dd t \ : \ {(\rho, \psi) \in \CE'_{1}(\bar\rho_0,\bar\rho_1)} \bigg\}\;.
\end{align*}
\end{proposition}

Note that the assumptions above on the jump kernel $J$ are satisfied
by the kernel $J_\a$ associated to the fractional Laplacian.

\begin{proof}
  The inequality `$\leq$' follows easily by noting that the infimum
  in the definition of $\cW$ is taken over a larger set. Indeed, given
  a pair $(\rho,\psi)\in\CE'_{1}(\bar\rho_0,\bar\rho_1)$ such that
  $\int_0^1\cA'(\rho_t,\psi_t)\dd t$ is finite we set $\mu_t=\rho_tm$
  and define $\bnu_t\in\cM_{loc}(G)$ by $\bnu_t(\dd x,\dd
  y)=\dgrad\psi_t(x,y)\hat\rho_t(x,y)J(x,\dd y)m(\dd x)$. Then
  obviously we have $\cA'(\rho_t,\psi_t)=\cA(\mu_t,\bnu_t)$ and it is
  easily checked using Lemma \ref{lem:integrability} that
  $(\mu,\bnu)\in\CE_{1}(\bar\mu_0,\bar\mu_1)$.

  Let us now prove the opposite inequality `$\geq$'. To this end, note
  that by a reparametrization argument similar to Lemma
  \ref{lem:equivcharacter} the square root of the infimum on the right
  hand side coincides with
  \begin{align*}
    \inf\left\{\int_0^T\sqrt{\cA'(\rho_t,\psi_t)}\dd t\ :\
    (\rho,\psi)\in\CE'_{T}(\bar\rho_0,\bar\rho_1)\right\}\ .
  \end{align*}
  We set $\mu^{i,\eps}_t:=P_{t}[\bar\mu_i]=\rho^{i,\eps}_tm$ and
  $\psi^{i,\eps}_t=\log\rho^{i,\eps}_t$ for $i=0,1$ and
  $t\in(0,\eps]$. It is easily checked, that the pair
  $(\rho^{i,\eps},\psi^{i,\eps})$ belongs to
  $\CE'_{\eps}(\bar\rho_i,\rho^{i,\eps}_1)$. Using the monotonicity
  of $\cI$ under convolution as in the proof of Claim
  \ref{cl:EVIforJP} we infer that
  \begin{align*}
    L^{i,\eps}~:=~\int_0^\eps\sqrt{\cA'(\rho_t^{i,\eps},\psi_t^{i,\eps})}\dd
    t~=~\int_0^\eps\sqrt{\cI(\mu^{i,\eps}_t)}\dd
    t~\leq~\eps\sqrt{\cI(\bar\mu_i)}\ .
  \end{align*}
  Now let $(\mu,\bnu)\in\CE_{1}(\bar\mu_0,\bar\mu_1)$ be a geodesic
  and set $\mu^\eps_t:=P_\eps[\mu_t]=\rho^\eps_tm$. Proposition
  \ref{prop:metricderivative} and the proof of Proposition
  \ref{prop:monotonconvolutionW} show that the curve
  $t\mapsto\mu^\eps_t$ is absolutely continuous w.r.t. $\cW$ and thus
  there is a family of optimal velocity measures $\tilde\bnu^\eps$. By
  Proposition \ref{prop:tangentspace-density} we have that
  $\tilde\bnu_t^\eps=w^\eps_t\hat\rho^\eps_tJm$ where $w^\eps_t$
  belongs to $T_\rho^\eps$. Note that $\rho_t^{\eps}>0$ by Assumption
  \ref{ass:regularity} and thus $\hat\rho_t^{\eps}>0$ for all
  $t\in(0,1)$ and moreover $j>0$. Hence it is easily checked any limit
  of discrete gradients in $L^2$ w.r.t. the measure
  $\hat\rho_t^{\eps}Jm(\dd x,\dd y)=\hat\rho_t^{\eps}(x,y)j(y-x)\dd
  x\dd y$ coincides again a.e. with a discrete gradient. Thus we have
  $w_t^\eps=\dgrad\psi_t^\eps$ a.e. for a suitable function
  $\psi^\eps:(0,1)\times\R^d\to\R$. Now observe that
  $(\rho^\eps,\psi^\eps)\in\CE'
_{1}(\rho^\eps_0,\rho^\eps_1)$ and
  \begin{align*}
    L^\eps~&:=~\int_0^1\sqrt{\cA'(\rho^\eps_t,\psi^\eps_t)}\dd
    t~=~\int_0^1\sqrt{\cA(\mu^\eps_t,\tilde\bnu^\eps_t)}\dd
    t\\~&\leq~\int_0^1\sqrt{\cA(\mu_t,\bnu_t)}\dd
    t~=~\cW(\bar\mu_0,\bar\mu_1)\ ,
  \end{align*}
  where we have used Proposition \ref{prop:monotonconvolution} in the
  second line. Finally we concatenate the three curves
  $(\rho^{0,\eps},\psi^{0,\eps}),(\rho^{\eps},\psi^{\eps})$ and
  $(\rho^{1,\eps},\psi^{1,\eps})$ to obtain a curve
  $(\tilde\rho^{\eps},\tilde\psi^{\eps})\in\CE'_{1+2\eps}(\bar\rho_0,\bar\rho_1)$
  which satisfies
  \begin{align*}
    \int_0^{1+2\eps}\sqrt{\cA'(\tilde\rho_t^\eps,\tilde\psi_t^\eps)}\dd t~&=~L^{0,\eps}+L^\eps+L^{1,\eps}~\\
&\leq~\cW(\bar\mu_0,\bar\mu_1) + \eps(\cI(\bar\mu_0)+\cI(\bar\mu_1))\ .
  \end{align*}
  Letting $\eps$ go to zero now yields the claim.
\end{proof}

\bibliographystyle{plain}
\bibliography{literature}

\begin{thebibliography}{10}

\bibitem{AGS05}
L.~Ambrosio, N.~Gigli, and G.~Savar\'{e}.
\newblock {\em {Gradient flows in metric spaces and in the space of probability
  measures}}.
\newblock Lectures in Mathematics ETH. Birkh\"{a}user, Z\"{u}rich, 2005.

\bibitem{AGS11a}
L.~Ambrosio, N.~Gigli, and G.~Savar{\'e}.
\newblock Calculus and heat flow in metric measure spaces and applications to
  spaces with {R}icci bounds from below.
\newblock {\em Preprint at arXiv:1106.2090}, 2011.

\bibitem{ASZ09}
L.~Ambrosio, G.~Savar{\'e}, and L.~Zambotti.
\newblock Existence and stability for {F}okker-{P}lanck equations with
  log-concave reference measure.
\newblock {\em Probab. Theory Related Fields}, 145(3-4):517--564, 2009.

\bibitem{App04}
D.~Applebaum.
\newblock {\em L\'evy processes and stochastic calculus}, volume~93 of {\em
  Cambridge studies in advanced mathematics}.
\newblock Cambridge University Press, 2004.

\bibitem{BE85}
D.~Bakry and Michel {\'E}mery.
\newblock Diffusions hypercontractives.
\newblock In {\em S\'eminaire de probabilit\'es {XIX}}, volume 1123 of {\em
  Lecture Notes in Math.}, pages 177--206. Springer, Berlin, 1985.

\bibitem{BBCK09}
M.~Barlow, R.~Bass, Z.-G. Chen, and M.~Kassmann.
\newblock Non-local {D}irichlet forms and symmetric jump processes.
\newblock {\em Trans. Amer. Math. Soc.}, 361(4):1963–1999, 2009.

\bibitem{BB00}
J.-D. Benamou and Y.~Brenier.
\newblock A computational fluid mechanics solution to the {M}onge-{K}antorovich
  mass transfer problem.
\newblock {\em Numer. Math.}, 84(3):375--393, 2000.

\bibitem{Bu89}
G.~Buttazzo.
\newblock {\em {Semicontinuity, relaxation and integral representation in the
  calculus of variations}}.
\newblock Pitman Research Notes in Mathematics Series. Longman Scientific and
  Technical, Harlow, 1989.

\bibitem{CS11}
L.~Caffarelli and L.~Silvestre.
\newblock The {E}vans-{K}rylov theorem for non local fully non linear
  equations.
\newblock {\em Ann. of Math.}, 174(2):1163--1187, 2011.

\bibitem{CHLZ11}
S.-N. Chow, W.~Huang, Y.~Li, and H.~Zhou.
\newblock Fokker-{P}lanck equations for a free energy functional or {M}arkov
  process on a graph.
\newblock {\em Arch. Ration. Mech. Anal.}, 203(3):969--1008, 2012.

\bibitem{DS08}
S.~Daneri and G.~Savar{\'e}.
\newblock Eulerian calculus for the displacement convexity in the {W}asserstein
  distance.
\newblock {\em SIAM J. Math. Anal.}, 40(3):1104--1122, 2008.

\bibitem{DNS09}
J.~Dolbeault, B.~Nazaret, and G.~Savar\'{e}.
\newblock A new class of transport distances between measures.
\newblock {\em Calc. Var. Partial Differential Equations}, 34(2):193--231,
  2009.

\bibitem{Er10}
M.~Erbar.
\newblock The heat equation on manifolds as a gradient flow in the
  {W}asserstein space.
\newblock {\em Ann. Inst. Henri Poincar\'e Probab. Stat.}, 46(1):1--23, 2010.

\bibitem{EM11}
M.~Erbar and J.~Maas.
\newblock Ricci curvature of finite {M}arkov chains via convexity of the
  entropy.
\newblock {\em Preprint at arXiv: 1111.2687}, 2011.

\bibitem{FSS10}
S.~Fang, J.~Shao, and K.-Th. Sturm.
\newblock Wasserstein space over the {W}iener space.
\newblock {\em Probab. Theory Related Fields}, 146(3-4):535--565, 2010.

\bibitem{Gi10}
N.~Gigli.
\newblock On the heat flow on metric measure spaces: existence, uniqueness and
  stability.
\newblock {\em Calc. Var. Partial Differential Equations}, 39:101--120, 2010.

\bibitem{GKO10}
N.~Gigli, K.~Kuwada, and S.-I. Ohta.
\newblock Heat flow on {A}lexandrov spaces.
\newblock {\em Preprint at arXiv:1008.1319}, 2010.

\bibitem{JKO98}
R.~Jordan, D.~Kinderlehrer, and F.~Otto.
\newblock The variational formulation of the {F}okker-{P}lanck equation.
\newblock {\em SIAM J. Math. Anal.}, 29(1):1--17, 1998.

\bibitem{LV09}
J.~Lott and C.~Villani.
\newblock Ricci curvature for metric-measure spaces via optimal transport.
\newblock {\em Ann. of Math. (2)}, 169(3):903--991, 2009.

\bibitem{Ma11}
J.~Maas.
\newblock Gradient flows of the entropy for finite {M}arkov chains.
\newblock {\em J. Funct. Anal.}, 261(8):2250 -- 2292, 2011.

\bibitem{McC97}
R.~McCann.
\newblock A convexity principle for interacting gases.
\newblock {\em Adv. Math.}, 128(1):153--179, 1997.

\bibitem{Mie11b}
A.~Mielke.
\newblock Geodesic convexity of the relative entropy in reversible {M}arkov
  chains.
\newblock {\em Preprint}, 2011.

\bibitem{OhSt09}
S.-I. Ohta and K.-Th. Sturm.
\newblock Heat flow on {F}insler manifolds.
\newblock {\em Comm. Pure Appl. Math.}, 62(10):1386--1433, 2009.

\bibitem{O01}
F.~Otto.
\newblock The geometry of dissipative evolution equations: the porous medium
  equation.
\newblock {\em Comm. Partial Differential Equations}, 26(1-2):101--174, 2001.

\bibitem{OV00}
F.~Otto and C.~Villani.
\newblock Generalization of an inequality by {T}alagrand and links with the
  logarithmic {S}obolev inequality.
\newblock {\em J. Funct. Anal.}, 173(2):361--400, 2000.

\bibitem{St06}
K.-Th. Sturm.
\newblock On the geometry of metric measure spaces. {I} and {II}.
\newblock {\em Acta Math.}, 196(1):65--177, 2006.

\bibitem{Vil09}
C.~Villani.
\newblock {\em Optimal transport, Old and new}, volume 338 of {\em Grundlehren
  der Mathematischen Wissenschaften}.
\newblock Springer-Verlag, Berlin, 2009.

\end{thebibliography}

\end{document}